\documentclass[reqno]{amsart}

\usepackage{amsmath,amsthm,amssymb,bm}
\usepackage{hyperref,ytableau}
\ytableausetup{boxsize=1.2em}
\usepackage{a4wide}
\usepackage{cleveref}

\usepackage{graphicx,color}
\usepackage{tikz}
\numberwithin{equation}{section}

\newtheorem{thm}{Theorem}[section]
\newtheorem{lem}[thm]{Lemma}
\newtheorem{prop}[thm]{Proposition}

\theoremstyle{definition}
\newtheorem{exam}[thm]{Example}
\newtheorem{defn}[thm]{Definition}

\newtheorem{problem}[thm]{Problem}

\crefname{lem}{Lemma}{Lemmas}
\crefname{thm}{Theorem}{Theorems}
\crefname{exam}{Example}{Examples}
\crefname{prop}{Proposition}{Propositions}
\crefname{question}{Question}{Questions}
\crefname{defn}{Definition}{Definitions}
\crefname{conj}{Conjecture}{Conjectures}
\crefname{figure}{Figure}{Figures}
\crefname{cor}{Corollary}{Corollaries} 
\crefformat{equation}{(#2#1#3)}
\Crefformat{equation}{Equation #2(#1)#3}

\AddToHook{env/lem/begin}{\crefalias{thm}{lem}}
\AddToHook{env/prop/begin}{\crefalias{thm}{prop}}
\AddToHook{env/cor/begin}{\crefalias{thm}{cor}}
\AddToHook{env/exam/begin}{\crefalias{thm}{exam}}
\AddToHook{env/defn/begin}{\crefalias{thm}{defn}}
\AddToHook{env/conj/begin}{\crefalias{thm}{conj}}
\AddToHook{env/question/begin}{\crefalias{thm}{question}}
\AddToHook{env/problem/begin}{\crefalias{thm}{problem}}
\AddToHook{env/remark/begin}{\crefalias{thm}{remark}}

\newcommand\maxDes{\operatorname{maxDes}}
\newcommand\LR{\operatorname{LR}}
\newcommand\spin{\operatorname{spin}}
\newcommand\GL{\operatorname{GL}}
\newcommand\rect{\operatorname{rect}}
\newcommand\st{\operatorname{st}}
\newcommand\RSK{\operatorname{RSK}}
\newcommand\shuffle{\operatorname{shuffle}}
\newcommand\sh{\operatorname{sh}}
\newcommand\ch{\operatorname{ch}}
\newcommand\maj{\operatorname{maj}}
\newcommand\Des{\operatorname{Des}}
\newcommand\QQ{\mathbb{Q}}
\newcommand{\CC}{\mathbb{C}}
\newcommand{\ZZ}{\mathbb{Z}}
\newcommand\LL{\mathcal{L}}

\newcommand\SSYT{\operatorname{SSYT}}
\newcommand\SYT{\operatorname{SYT}}
\newcommand\YDT{\operatorname{YDT}}

\newcommand\wt{\operatorname{wt}}

\renewcommand\emph[1]{\textcolor{blue}{\it #1}}

\usepackage{xcolor}
\usepackage{textcomp} \usepackage{listings}
\usepackage{amssymb}
\usepackage{upquote}
\usepackage[T1]{fontenc}
\usepackage{multirow}
\usepackage{tikz-cd}

\theoremstyle{definition}

\title{Thrall's problem for two rows}

\author{JiSun Huh}
\address{Department of Mathematics, University of Seoul, Seoul 02504, South Korea}
\email{hyunyjia@yonsei.ac.kr}

\author{Woo-Seok Jung}
\address{Department of Mathematics, University of Seoul, Seoul 02504, South Korea}
\email{jungws@uos.ac.kr}

\author{Jang Soo Kim}
\address{Department of Mathematics,
Sungkyunkwan University (SKKU), Suwon 16419, South Korea}
\email{jangsookim@skku.edu}

\author{Meesue Yoo}
\address{
Department of Mathematics, Chungbuk National University, Cheongju 28644,
South Korea}
\email{meesueyoo@chungbuk.ac.kr}


\begin{document}

\begin{abstract}
In this paper, we study Thrall's problem for the higher Lie modules $L_\lambda$. Our main result provides a tableau-theoretic description of the Schur expansion of the character of $L_\lambda$ when $\lambda$ has two rows, thereby solving Thrall's problem in this case. This formula is expressed in terms of standard Young tableaux with major index congruence conditions and a spin-parity condition defined through bijections with Yamanouchi domino tableaux. We also obtain tableau formulas for hook shapes and partitions with distinct parts, and these results extend to all partitions in which each part greater than $2$ occurs at most twice.
\end{abstract}

\maketitle


\section{Introduction}

In this paper, we study Thrall's problem for the higher Lie modules.
We begin by describing the problem. See
\Cref{sec:preliminaries} for the undefined terms.

Consider the \( \CC \)-vector space \(V=\CC^N\) and let
\(T(V)=\bigoplus_{m\ge 0} V^{\otimes m} \)
be the tensor algebra on \(V\), graded by tensor degree with \(V\) in
degree \(1\). The \emph{free Lie algebra} \( \operatorname{Lie}(V) \)
on \(V\) is the Lie subalgebra of \( T(V) \) generated by \(V\) under
the commutator bracket \([x,y]=xy-yx\). Since the bracket is
homogeneous with respect to the tensor grading, we have
\(\operatorname{Lie}(V)=\bigoplus_{m\ge 1} \LL_m(V) \),
where \(\LL_m(V):=\operatorname{Lie}(V)\cap V^{\otimes m} \) is the
homogeneous component of degree \( m \).
For a partition \(\lambda \), the \emph{higher Lie module}
\( \LL_\lambda(V) \) indexed by \(\lambda\) is defined by
\begin{equation}\label{eq:L_la(V)}
\LL_\lambda(V):=\bigotimes_{i\ge 1}\mathrm{Sym}^{m_i(\lambda)}(\LL_i(V)),
\end{equation}
where \( m_i(\lambda) \) denotes the number of parts equal to \( i \)
in \( \lambda \). 
We also use the convention that \(\LL_\emptyset(V)\) denotes the trivial
one-dimensional \(\GL(V)\)-module.

By the Poincar\'e--Birkhoff--Witt theorem, the tensor algebra is identified as a graded \(\GL(V)\)-module with the symmetric algebra on the free Lie algebra:
\[
  T(V)\cong  \operatorname{Sym}(\operatorname{Lie}(V))
  \cong \bigoplus_{\lambda} \LL_\lambda(V),
\]
where the direct sum is over all partitions \(\lambda\).
Thus the higher Lie modules arise naturally from the isomorphism above. For brevity,
we write
\[
L_\lambda:=\LL_\lambda(V),
\qquad
L_n:=L_{(n)}.
\]

For \(g\in \GL(V)\), let \(x_1,\dots,x_N\) be the eigenvalues of
\(g\). Then the character \( \ch(L_\lambda) \) of the polynomial
\(\GL(V)\)-module \(L_\lambda\) is a symmetric polynomial in
\(x_1,\dots,x_N\). In 1942, Thrall \cite{Thrall1942} asked for the
Schur decomposition of the higher Lie modules \(L_\lambda\).
Equivalently, the problem can be stated as follows.

\begin{problem}[Thrall's Problem]
  For a partition \(\lambda\) of \(n\), find a combinatorial
  interpretation for the Schur coefficients of the characters
  \(\ch(L_\lambda)\) in the limit \(N\to\infty\):
\begin{equation}\label{eq:4}
\ch(L_\lambda)=\sum_{\mu\vdash n} c_{\lambda,\mu}s_\mu.
\end{equation}
\end{problem}

Thrall's problem is notoriously difficult, and it has been solved only
in the cases where \( \lambda \) has one row, one column, or two
equal columns. We review these results below.

For the one-column and two-column cases, we have
\begin{align}
\label{eq:1^n}  \ch(L_{(1^n)}) &= s_n,\\
\label{eq:2^n}   \ch(L_{(2^n)}) &= \sum_\mu  s_\mu,
\end{align}
where the sum is over all partitions \( \mu\vdash 2n \) such that
each column has even length. See \cite[Remark~2.17]{Ahlbach2018} for
more details. The following theorem answers Thrall's problem for the
one-row case. Here, \( \SYT(\mu) \) is the set of standard Young
tableaux of shape \( \mu \) and \( \maj(T) \) is the major index of
\( T \).

\begin{thm}[Kra\'skiewicz--Weyman \cite{KW}]\label{thm:KW}
For every \(n\ge 1\), we have
\begin{equation}\label{eq:13}
\ch(L_n)=\sum_{\mu\vdash n} |\{T\in\SYT(\mu): \maj(T)\equiv 1 \pmod n\}| s_\mu.
\end{equation}
\end{thm}

\Cref{thm:KW} first appeared in a 1987 preprint by Kra\'skiewicz and
Weyman, which was later published in \cite{KW}. An equivalent formula had been
obtained by Klyachko in 1974 \cite{Klyachko1974}:
\[
\ch(L_n)
= \sum_{\nu\vdash n} \left( \frac{1}{n}\sum_{d\mid n} \mu(d)\,\chi^\nu(\tau^{n/d}) \right) s_\nu,
\]
where \(\tau\) is the long cycle in \(S_n\), \(\chi^\nu\) is the
irreducible character of \(S_n\) indexed by \(\nu\), and \(\mu\) is
the M\"obius function. See \cite[Section~2.6]{Ahlbach2018} for the
connection between the two formulas.

The starting point of this paper is the observation that the
coefficient of \( s_\mu \) in \eqref{eq:13} is given by the number of
standard Young tableaux of shape \( \mu \) satisfying a certain
condition. Motivated by this, we define a
\emph{\( (\lambda,\mu) \)-Thrall subset} to be a subset of
\( \SYT(\mu) \) whose cardinality is equal to the coefficient
\( c_{\lambda,\mu} \) in \eqref{eq:4}.

Note that given \( \lambda\vdash n \), finding a
\( (\lambda,\mu) \)-Thrall subset for each \( \mu\vdash n \) gives an
answer to Thrall's problem for \( L_\lambda \). For example, the set
in \eqref{eq:13} is an \( ((n),\mu) \)-Thrall subset. Observe that
\eqref{eq:1^n} and \eqref{eq:2^n} can be rewritten as
\begin{align}\label{eq:one-col}
  \ch(L_{(1^n)})&= \sum_{\mu\vdash n} |\{T\in \SYT (\mu):\Des(T)=\emptyset\}| s_\mu,\\
\label{eq:two-col}
  \ch(L_{(2^n)})&= \sum_{\mu\vdash 2n} |\{T\in \SYT (\mu):\{1,3,\dots,2n-1\}\subseteq\Des(T)= \maxDes(\mu)\}| s_\mu,
\end{align}
where \( \Des(T) \) is the descent set of \( T \), and writing \( \mu' \) for the conjugate partition of \( \mu \),
\begin{equation}\label{eq:11}
  \maxDes(\mu) = [2n-1]\setminus \{\mu'_1 + \cdots + \mu'_j: j\in [\mu_1-1]  \}.
\end{equation}
Hence, the sets given in \eqref{eq:one-col} and \eqref{eq:two-col} are
a \( ((1^n),\mu) \)-Thrall subset and a \( ((2^n),\mu) \)-Thrall
subset, respectively. The following argument shows that
\( (\lambda,\mu) \)-Thrall subsets always exist.

By the Poincar\'e--Birkhoff--Witt theorem and Schur--Weyl duality, we have
\[
  \bigoplus_{\lambda\vdash n} L_\lambda \cong V^{\otimes n} \cong
  \bigoplus_{\mu\vdash n} V_\mu^{\oplus |\SYT(\mu)|},
\]
where \( V_\mu \) is the irreducible representation of
\( \GL(V) \) indexed by \( \mu \). Hence,
\[
  \sum_{\lambda\vdash n} \ch(L_\lambda) = \sum_{\mu\vdash n} |\SYT(\mu)| s_\mu,
\]
and we obtain that for a fixed \( \mu\vdash n\),
\[
  \sum_{\lambda\vdash n} c_{\lambda,\mu} = |\SYT(\mu)|.
\]
This implies that for any partitions \( \lambda \) and \( \mu \) of
\( n \), we have \( c_{\lambda,\mu} \le |\SYT(\mu)| \). Therefore,
there always exists a subset of \( \SYT(\mu) \) with cardinality
\( c_{\lambda,\mu} \). In \Cref{sec:hook-distinct}, we will prove the
stronger bound \( c_{\lambda,\mu} \le |\SYT_\lambda(\mu)| \) for
a certain subset \( \SYT_\lambda(\mu) \) of \( \SYT(\mu) \).

The goal of this paper is to find \( (\lambda,\mu) \)-Thrall subsets
for several classes of \( \lambda \): hook shapes, partitions with
distinct parts, and partitions with two rows. Most generally, we find a \( (\lambda,\mu) \)-Thrall subset when
\( \lambda \) is a partition such that every part greater than
\( 2 \) appears at most twice.

The key tool in this paper is the notion of the block-major index
\( \maj_{\lambda,i}(T) \) of a standard Young tableau \( T \) with
respect to a partition \( \lambda \), which is a refinement of
\( \maj(T) \). We then define
\[
    \SYT_\lambda(\mu) := \left\{T\in \SYT(\mu) :
      \maj_{\lambda,i}(T)\equiv 1 \pmod{\lambda_i} \text{ for all }i \right\}.
\]
See \Cref{sec:hook-distinct} for the precise definitions. In
\Cref{thm:distinct}, we show that if \( \lambda \) has distinct parts,
then
\[
\ch(L_\lambda) = \sum_{\mu\vdash n} |\SYT_\lambda(\mu)| s_\mu.
\]

The main result of the paper is \Cref{thm:nn-final-single-sum}, which
solves Thrall's problem for \( \lambda=(n,n) \) as follows:
\[
\ch(L_{(n,n)})
= \sum_{\mu\vdash 2n} |\SYT_{(n,n)}^{<}(\mu)\sqcup \SYT_{(n,n)}^{\spin}(\mu)| s_\mu.
\]
Here, \(\SYT_{(n,n)}^{<}(\mu)\) and \(\SYT_{(n,n)}^{\spin}(\mu)\) are
subsets of \(\SYT_{(n,n)}(\mu)\) with certain conditions on the
subtableaux with the first \( n \) entries and the last \( n \)
entries. In particular, the description of
\(\SYT_{(n,n)}^{\spin}(\mu)\) involves a spin statistic coming from
the work of Carr\'e and Leclerc \cite{Carre1995} on domino tableaux.
To prove the main result, we use various results on semistandard Young
tableaux and Yamanouchi domino tableaux due to Carr\'e and Leclerc
\cite{Carre1995} and van Leeuwen \cite{vanLeeuwen1999, vanLeeuwen2000,
  vanLeeuwen2001}.

This paper is organized as follows. In \Cref{sec:preliminaries}, we
recall basic definitions and results on partitions, tableaux,
symmetric functions, higher Lie modules, and the Carr\'e--Leclerc
formula for the Schur expansion of the plethysm \( h_2[s_\lambda] \) using
Yamanouchi domino tableaux. In \Cref{sec:hook-distinct}, we introduce
block-major indices and refined \( (\lambda,\mu) \)-Thrall subsets. We
prove that refined \( (\lambda,\mu) \)-Thrall subsets always exist,
and provide a method to construct refined \( (\lambda,\mu) \)-Thrall
subsets using those of the rectangular blocks in \( \lambda \). We also find
refined \( (\lambda,\mu) \)-Thrall subsets when \( \lambda \) is a
partition with distinct parts or a hook shape. In
\Cref{sec:equal-parts}, we find refined \( (\lambda,\mu) \)-Thrall
subsets when \( \lambda=(n,n) \), and generalize them to partitions in
which each part greater than \( 2 \) occurs at most twice. In
\Cref{sec:further-directions}, we propose open problems on arbitrary
rectangular shapes and on a direct description of the spin statistic.

\section{Preliminaries}\label{sec:preliminaries}

In this section, we review the basic definitions and known results
that will be used throughout this paper.

\subsection{Basic definitions}

For integers \( a \) and \( b \), let \( [a,b] \) be the set of
integers \( i \) with \( a\le i\le b \), and let \([a]:=[1,a]\). We
denote by \( S_n \) the set of permutations of \( [n] \).

A \emph{partition} \(\lambda=(\lambda_1,\dots,\lambda_r)\) of \(n\), denoted
\(\lambda\vdash n\), is a weakly decreasing sequence of positive integers
summing to \( n \). Each \(\lambda_i\) is called a \emph{part} and the
number of parts of \(\lambda\) is called the \emph{length} of \(\lambda\),
denoted by \(\ell(\lambda)\). We denote by \( m_i(\lambda) \) the number
of parts equal to \( i \) in \( \lambda \). For two partitions
\( \lambda \) and \( \mu \), we define \( \lambda\cup \mu \) to be the
partition with \( m_i(\lambda)+m_i(\mu) \) parts equal to \( i \) for
each \( i\ge1 \).

A \emph{cell} is a pair \( (i,j) \) of positive integers. The
\emph{Young diagram} of \( \lambda \) is the set
\(\{(i,j): 1\le i\le \ell(\lambda),\, 1\le j\le \lambda_i\}\). We will
identify the partition \(\lambda\) with its Young diagram. We draw
Young diagrams in English notation, with rows indexed from top to
bottom and columns from left to right. The \emph{conjugate} of \( \lambda \), denoted by \( \lambda' \), 
is the partition whose Young diagram
is \( \{(j,i):(i,j)\in \lambda\} \). For two partitions \( \lambda \)
and \( \mu \) with \( \lambda\subseteq\mu \), the \emph{skew shape}
\( \mu/\lambda \) is the set-theoretic difference \( \mu-\lambda \) of
their Young diagrams. A partition \( \mu \) is also considered as a
skew shape \( \mu/\emptyset \), called a \emph{straight shape}.

For any set \( A \) of cells, a \emph{tableau} \( T \) of shape
\( A \) is a function \( T:A\to \ZZ_{\ge1} \). The value \( T(i,j) \) is called the
\emph{entry} of the cell \( (i,j)\in A \). We denote by \( \sh(T) \) the shape \( A \) of the
tableau \( T \). We say that \( T \) is \emph{column-strict} if the
following conditions hold:
\begin{enumerate}
  \item If \( (i,j) \) and \( (i',j') \) are cells in \( \sh(T) \)
    with \( i\le i' \) and \( j\le j' \), then \( T(i,j)\le T(i',j') \).
  \item If \( (i,j) \) and \( (i',j) \) are cells in \( \sh(T) \)
    with \( i< i' \), then \( T(i,j)< T(i',j) \).
\end{enumerate}

A \emph{semistandard Young tableau} is a column-strict tableau whose
shape is a skew shape. We define the \emph{weight} of a semistandard Young
tableau \(T\) to be the sequence \(\wt(T)=(\wt_1(T),\wt_2(T),\dots)\),
where \(\wt_i(T)\) is the number of occurrences of \(i\) in \(T\). 
If \(T\) has \(n\) cells and each of the entries \(1,2,\dots,n\) appears exactly once, we call \(T\) a
\emph{standard Young tableau}. The set of semistandard Young tableaux (respectively,
standard Young tableaux) of shape \(\mu/\lambda\) is denoted by
\(\SSYT(\mu/\lambda)\) (respectively, \(\SYT(\mu/\lambda)\)). We also define
\[
\SYT(n):=\bigsqcup_{\mu\vdash n}\SYT(\mu).
\]

For a partition \(\mu\),
the \emph{complete homogeneous symmetric function}
\(h_\mu\) is defined by \(h_\mu=\prod_{i\ge1} h_{\mu_i}\), where
\(h_0=1\) and \(h_k=\sum_{i_1\le\cdots \le i_k}x_{i_1}\cdots x_{i_k}\)
for \(k\ge 1\). The \emph{power sum symmetric function} \(p_\mu\) is
defined by \(p_\mu=\prod_{i\ge1} p_{\mu_i}\), where \(p_0=1\) and
\(p_k =\sum_{i\ge1} x_i ^k\) for \(k\ge 1\). For a skew shape
\( \mu/\lambda \), the \emph{Schur function} \(s_{\mu/\lambda}\) is defined by
\[
s_{\mu/\lambda} =\sum_{T\in\SSYT(\mu/\lambda)}x^T,
\]
where \(x^T=\prod_{i\ge1} x_i ^{\wt_i(T)}\).

For a semistandard Young tableau \(S\) of any skew shape, the
\emph{standardization} \(\st(S)\) of \(S\) is the standard Young tableau
obtained from \( S \) by replacing the leftmost occurrence of the smallest entry by
\( 1 \), the leftmost occurrence of the smallest entry among the remaining entries by
\( 2 \), and so on. If \(T\) is a standard Young tableau of any skew
shape with \( n \) cells and \(I\subseteq [n]\) is a set of
consecutive integers, then \(T_I\) denotes the subtableau obtained by
restricting \(T\) to the entries in \(I\).

\begin{defn}
  For a column-strict tableau \( T \) of shape \( A \), a \emph{jeu de
    taquin slide} is an operation that interchanges a cell \( (i,j) \)
  not in \( A \) with a cell \( (i',j')\in \{(i+1,j),(i,j+1)\} \) in
  \( A \), carrying the entry of \( (i',j') \) with it, such that the resulting
  tableau is still column-strict. If \( T\in\SSYT(\mu/\lambda) \), the
  \emph{rectification} \( \rect(T) \) of \( T \) is defined to be the
  semistandard Young tableau of a straight shape obtained from \( T \)
  by applying jeu de taquin slides until this is no longer possible.
\end{defn}

The rectification \( \rect(T) \) is independent of the sequence of jeu
de taquin slides, and hence is well-defined
(cf.~\cite[A1.2]{EC2}). The following definition will be useful
in this paper.

\begin{defn}\label{def:1}
  For \( T\in \SYT(n) \) and a subset \( I \subseteq [n] \) of
  consecutive integers, we define
  \[
    T^{I} := \rect\bigl(\st(T_I)\bigr).
  \]
\end{defn}

Note that if \( |I|=k \), then \( T^I \in \SYT(k) \). In particular,
if \( I=[k] \), then \( T^I = T_I \).

For a standard Young tableau \(T\) with \( n \) cells, the \emph{descent set} of \( T \) is
\[
\Des(T):=\{\, i\in [n-1]:i+1 \text{ lies in a lower row than } i\,\},
\]
and its \emph{major index} is defined by 
\[
\maj(T):=\sum_{i\in \Des(T)} i.
\]

For a permutation \(\pi\in S_n\), we define 
\[
\Des(\pi):=\{\,i\in[n-1]:\pi(i)>\pi(i+1)\,\}, \qquad 
  \maj(\pi) := \sum_{i\in \Des(\pi)} i.
\]
If \(w\in S_n\) and \(\RSK(w)=(P(w),Q(w))\) is the Robinson--Schensted correspondence (cf. \cite[Section~7.11]{EC2}), then
\[
\Des(w)=\Des(Q(w)),
\qquad
\Des(w^{-1})=\Des(P(w)).
\]

If \(w=w_1\cdots w_n\) is a word and \(I\subseteq [n]\) is a subset of
consecutive integers, then \(w_I\) denotes the subword of \(w\) consisting
of the letters in \(I\), read in their original order. When \(w\) is a
permutation, \(\st(w_I)\) is the word obtained by replacing the letters
of \(I\) by \(1,2,\dots,|I|\) in increasing order.

We shall use the following standard subword property of the
Robinson--Schensted correspondence; cf. \cite[Section~7.11, A1.2]{EC2}.

\begin{lem}\label{lem:rsk-subword}
Let \(w\in S_n\), and let \(I\subseteq [n]\) be a set of consecutive integers. Then
\[
P\bigl(\st(w_I)\bigr) = P(w)^I.
\]
\end{lem}

The characters of higher Lie modules can be expressed in terms of
plethysm. We briefly recall the notion of plethysm; see
\cite[p.~135]{Macdonald} for more details.

Let \(\Lambda_{\QQ}\) be the ring of symmetric functions over \(\QQ\).
For \(f\in \Lambda_{\QQ}\), let
\(\varphi_f:\Lambda_{\QQ}\to \Lambda_{\QQ} \) be the unique
\(\QQ\)-algebra homomorphism such that
\(\varphi_f(p_r)=p_r[f]:=f(x_1^r,x_2^r,\dots) \) for all \( r\ge1 \).
For \(g\in \Lambda_{\QQ}\), we define the \emph{plethysm} \(g[f]\) by
\(g[f]:=\varphi_f(g) \). Note that \( f[p_1] = f \) for any
\( f\in \Lambda_\QQ \).

By the definition \eqref{eq:L_la(V)} of the higher Lie module
\( \LL_\lambda(V) \), we have
\begin{equation}\label{eq:ch(L_la)}
\ch(L_\lambda) = \prod_{i\ge1} h_{m_i(\lambda)}[\ch(L_i)].
\end{equation}

\subsection{Useful results}

For the rest of this section, we review some known results that will
be used throughout the paper.

Given a subset \(D\) of \([n-1]\), we define the \emph{fundamental quasisymmetric functions} \(F_{n, D}\) by 
\[
F_{n,D}=\sum_{\substack{i_1\le \cdots \le i_n \\ i_j <i_{j+1} \text{ if } j\in D}} x_{i_1}\cdots x_{i_n}.
\]
Then for \(\mu\vdash n\), the Schur function \(s_\mu\) expands in the
fundamental basis as
\begin{equation}\label{eq:schur-fundamental}
 s_\mu = \sum_{T\in\SYT(\mu)} F_{n,\Des(T)} .
\end{equation}
Combining \Cref{thm:KW} with \eqref{eq:schur-fundamental} and applying
the Robinson--Schensted correspondence, we obtain
\begin{equation}\label{eq:KW2}
\ch(L_n) = \sum_{\substack{\pi\in S_n \\\maj(\pi^{-1}) \equiv 1 \pmod{n}}} F_{n,\Des(\pi)}.
\end{equation}

For \(\pi\in S_n\) and \(\sigma\in S_m\), we write \(\shuffle(\pi,\sigma)\) for the set of permutations in \(S_{n+m}\) obtained by shuffling the words
\( \pi(1)\cdots \pi(n) \) and
\(n+\sigma(1),\dots,n+\sigma(m) \)
while preserving the relative order inside each word. We shall use the identity
\begin{equation}\label{eq:shuffle}
  F_{n,\Des(\pi)} F_{m,\Des(\sigma)}
  = \sum_{w\in \shuffle(\pi,\sigma)} F_{n+m,\Des(w)}
\end{equation}
to multiply \(\ch(L_n)\)'s later. 

We recall the Carr\'e--Leclerc formula for the symmetric
square \( h_2[s_{\lambda}] \) of a Schur function.

A \emph{domino} is a set of two adjacent cells. More precisely, a
\emph{vertical domino} is a set of the form \( \{(i,j),(i+1,j)\} \)
and a \emph{horizontal domino} is a set of the form
\( \{(i,j),(i,j+1)\} \). 
A \emph{domino tableau} of shape \(\mu/\lambda\) is a tiling of the Young
diagram of \(\mu/\lambda\) by dominoes, each labeled by a positive
integer. We regard the label of a domino as occupying both of its cells.
The labels are required to be weakly increasing along rows and strictly
increasing down columns.
The \emph{weight} of a domino tableau \( D \) is the sequence
\( (d_1,d_2,\dots) \), where \( d_i \) is the number of dominoes with label
\( i \).

A word \(w=w_1\cdots w_n\) of positive integers is called
\emph{Yamanouchi} if every suffix \(w_k\cdots w_n\) contains at least
as many letters equal to \(i\) as letters equal to \(i+1\), for each
\(i\ge 1\). The \emph{column reading word} of a domino tableau is 
the word obtained by reading the labels column by column, from bottom to top and from left to
right, with each horizontal domino read only once from its leftmost cell.
A \emph{Yamanouchi domino tableau} is a domino tableau whose
column reading word is Yamanouchi. We denote by
\( \YDT(\lambda,\mu) \) the set of Yamanouchi domino tableaux of shape
\( \lambda \) and weight \( \mu \).

\begin{defn}
Let \(D \in \YDT(\lambda,\mu)\). We denote by \(H(D)\) the number of horizontal dominoes of \(D\), and by \(V(D)\) the number of vertical dominoes of \(D\). We define the \emph{spin} of \(D\) to be
\[
\spin(D) = \frac{V(D)}{2}.
\]
\end{defn}

For a partition \( \lambda=(\lambda_1,\lambda_2,\dots) \), we define
\[
\lambda^\square := (2\lambda_1,2\lambda_1,2\lambda_2,2\lambda_2,\dots).
\]
Carr\'e and Leclerc showed the following combinatorial formula for the
Schur expansion of \( h_2[s_{\lambda}] \).

\begin{thm}\cite[Corollary~5.5]{Carre1995}\label{thm:CL}
For \(\lambda\vdash n\), we have
\[
h_2[s_{\lambda}]=\sum_{\mu \vdash 2n} |\{D\in\YDT(\lambda^\square,\mu): \spin(D) \equiv n \pmod{2}\}|\, s_{\mu}.
\]
\end{thm}

We note that in \cite[Corollary~5.5]{Carre1995} the result is stated
using the condition \( H(D) \equiv 0 \pmod 4 \) in place of
\( \spin(D) \equiv n \pmod 2 \). These two conditions are easily seen
to be equivalent. For completeness, we include the proof.

\begin{lem}\label{lem:spin-horizontal}
  Let \(\lambda\vdash n\), \(\mu\vdash 2n\), and
  \(D \in \YDT(\lambda^\square,\mu)\). Then
  \( H(D) \equiv 0 \pmod 4 \) if and only if
  \( \spin(D) \equiv n \pmod 2 \).
\end{lem}

\begin{proof}
  Let \(a_r\) be the number of vertical dominoes of \(D\) with cells
  in rows \(r\) and \(r+1\). Since every row of \(\lambda^\square\)
  has even length, each row contains an even number of cells belonging
  to vertical dominoes. Thus \(a_{r-1}+a_r\equiv 0 \pmod 2 \) for all
  \( r\ge1 \). Since \(a_0=0\), every \(a_r\) is even, and therefore
  \(V(D) \) is even. Hence \(\spin(D)=V(D)/2\in \ZZ\).

  The shape \(\lambda^\square\) has size \(4n\), so every domino
  tiling of this shape consists of \(2n\) dominoes. Hence
  \(H(D)+V(D)=2n\), and therefore \(H(D)=2(n-\spin(D)) \). Thus
  \(H(D)\equiv 0\pmod 4\) if and only if \(\spin(D)\equiv n\pmod 2\).
\end{proof}

We use the tableau switching map introduced in \cite{BSS1996}. 

\begin{defn}\label{def:tableau-switching-explicit}
  Let \( T_1 \) and \( T_2 \) be column-strict tableaux of shapes
  \( A_1 \) and \( A_2 \), respectively, where \( A_1 \) and \( A_2 \)
  are any disjoint sets of cells. A \emph{switch} is an operation that
  interchanges a cell \( (i,j) \) in \( A_1 \) with a cell
  \( (i',j')\in \{(i+1,j),(i,j+1)\} \) in \( A_2 \), carrying their entries with them,
  such that the resulting tableaux \( T_1 \) and \( T_2 \) are still
  column-strict.

  Let \( S\in \SSYT(\lambda/\nu) \) and \( T\in\SSYT(\mu/\lambda) \),
  where \( \nu\subseteq\lambda\subseteq\mu \) are partitions.
  The \emph{tableau switching} map \( X \) is defined by
  \( X(S,T) = (T',S') \), where \( S' \) and \( T' \) are the tableaux
  obtained from \( S \) and \( T \), respectively, by applying
  switches until this is no longer possible.
\end{defn}

See \Cref{fig:2} for an illustration of this procedure.

\begin{figure}
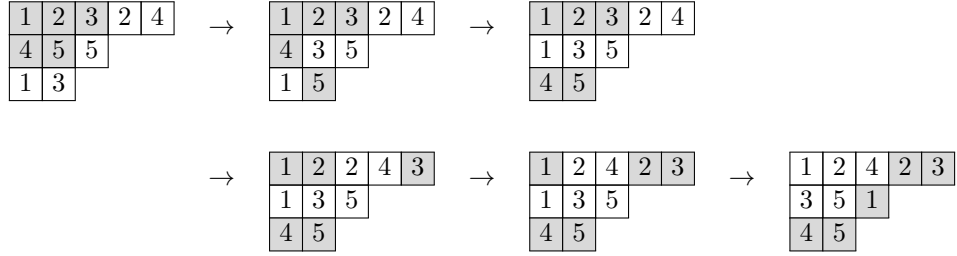

  \centering
  \begin{align*}
\begin{ytableau}
*(gray!30)1 &*(gray!30) 2 &*(gray!30) 3 & 2 & 4 \\
*(gray!30)4 &*(gray!30) 5 & 5 \\
1 & 3
\end{ytableau}
 \quad &\rightarrow \quad
 \begin{ytableau}
*(gray!30)1 &*(gray!30) 2 & *(gray!30)3 & 2 & 4 \\
*(gray!30)4 & 3 & 5 \\
1 & *(gray!30)5
\end{ytableau}
 \quad \rightarrow \quad
 \begin{ytableau}
*(gray!30)1 & *(gray!30)2 & *(gray!30)3 & 2 & 4 \\
1 & 3 & 5 \\
*(gray!30)4 & *(gray!30)5
\end{ytableau}\\
    & \\
 &\rightarrow \quad 
 \begin{ytableau}
*(gray!30)1 & *(gray!30)2 & 2 & 4 &*(gray!30) 3 \\
1 & 3 & 5 \\
*(gray!30)4 & *(gray!30)5
\end{ytableau}
 \quad \rightarrow \quad
 \begin{ytableau}
*(gray!30)1 & 2 & 4 & *(gray!30)2 & *(gray!30)3 \\
1 & 3 & 5 \\
*(gray!30)4 & *(gray!30)5
\end{ytableau}
 \quad \rightarrow \quad
 \begin{ytableau}
1 & 2 & 4 & *(gray!30)2 & *(gray!30)3 \\
3 & 5 & *(gray!30)1 \\
*(gray!30)4 & *(gray!30)5
\end{ytableau}
  \end{align*}
  \caption{An example of \( X(S,T)=(T',S') \).
    Cells originating from \( S \) are colored gray.}
  \label{fig:2}
\end{figure}

As in the case of rectification, tableau switching does not depend on the
sequence of switches. We shall use the following standard properties
of tableau switching; see \cite[Section~2]{vanLeeuwen2001} and
\cite{BSS1996}.

\begin{prop}\label{prop:ts}
  Suppose \(X(S,T)=(T',S') \), where \( S\in \SSYT(\lambda) \) and
  \( T\in \SSYT(\mu/\lambda) \). Then \(X(T',S')=(S,T) \), that is, the
  tableau switching map is involutive. Moreover, we have
  \( \rect(T)=T'\) and \(\rect(S')=S\).
\end{prop}

For a partition \(\lambda\), we denote by \(1_\lambda\) the unique
semistandard Young tableau of shape \(\lambda\) and weight \(\lambda\)
whose entries in row \(i\) are all equal to \(i\).

The \emph{reading word} of a semistandard Young tableau \(T\) is the
word obtained by reading each row from left to right, starting from
the bottom row and proceeding upward. A \emph{Littlewood--Richardson
  tableau} of shape \( \mu/\lambda \) and weight \(\nu\) is a
semistandard Young tableau of shape \( \mu/\lambda \) and weight \(\nu\)
whose reading word is Yamanouchi. We denote by \(\LR(\mu/\lambda,\nu) \)
the set of Littlewood--Richardson tableaux of shape \( \mu/\lambda \) and
weight \(\nu\).

For \(U\in \SYT(\lambda)\), we define
\[
  \SYT^{U}(\mu/\lambda) = \{ S\in \SYT(\mu/\lambda):\rect(S)=U \}.
\]
The following lemma is a consequence of van Leeuwen’s result
\cite[Corollary~2.5.2]{vanLeeuwen2001}. See \Cref{def:Psi_U} for a
detailed description of \(\Psi_U\).

\begin{lem}\label{lem:lr-rectification}
Fix \(U\in \SYT(\lambda)\). Then there is a bijection
\[
\Psi_U: \SYT^{U}(\mu/\lambda)
\to
\LR(\mu/\lambda,\lambda).
\]
\end{lem}

\section{Refined Thrall subsets and decomposition into rectangles}\label{sec:hook-distinct}

In this section, motivated by the result of Kra\'skiewicz--Weyman, we
introduce the refined \( (\lambda,\mu) \)-Thrall subsets. The key tool
is the notion of block-major indices, which are refinements of the
major index of a standard Young tableau. We provide a combinatorial
description of $\ch(L_\lambda)$ when \( \lambda \) has distinct parts
by finding refined \( (\lambda,\mu) \)-Thrall subsets. We show that
for any partitions \( \lambda \) and \( \mu \), such subsets always
exist. We also give a method to construct such a subset if refined
Thrall subsets for the rectangular blocks of \( \lambda \) are known.
As an application, we obtain a combinatorial description of
$\ch(L_\lambda)$ when \( \lambda \) is a hook.

We begin by introducing the block decomposition and the associated
major index statistics.

\begin{defn}\label{def:9}
Let $\lambda=(\lambda_1,\ldots,\lambda_r)\vdash n$. For $1\le i\le r$, we write
\[
B_{\lambda,i}=[\lambda_1+\cdots+\lambda_{i-1}+1,\lambda_1+\cdots+\lambda_i].
\]
For $T\in\SYT(n)$, define
\[
\Des_{\lambda,i}(T)=\{\, d-(\lambda_1+\cdots+\lambda_{i-1}) :d\in \Des(T), \, d,d+1\in B_{\lambda,i} \}.
\]
The \emph{$i$-th block-major index of $T$ with respect to $\lambda$},
denoted \( \maj_{\lambda,i}(T) \), is defined by
\[
\maj_{\lambda,i}(T)=\sum_{j\in \Des_{\lambda,i}(T)} j.
\]
\end{defn}

Recall the coefficient \( c_{\lambda,\mu} \) of \( s_\mu \) in the
Schur expansion of \( \ch(L_\lambda) \) in \eqref{eq:4}. We now
introduce the refined \( (\lambda,\mu) \)-Thrall subsets.

\begin{defn}\label{def:2}
  Let \( \lambda \) and \( \mu \) be partitions of \( n \). We define
  \[
    \SYT_\lambda(\mu) := \left\{T\in \SYT(\mu) :
      \maj_{\lambda,i}(T)\equiv 1 \pmod{\lambda_i} \text{ for all }i \right\}.
  \]
  A \emph{refined \( (\lambda,\mu) \)-Thrall subset} is a subset of
  \( \SYT_\lambda(\mu) \) whose cardinality is equal to
  \( c_{\lambda,\mu} \).
\end{defn}

Note that a refined \( (\lambda,\mu) \)-Thrall subset is a
\( (\lambda,\mu) \)-Thrall subset. In order to show that a refined
\( (\lambda,\mu) \)-Thrall subset always exists, we need some lemmas.

\begin{lem}\label{lem:rect-des}
  Let $\lambda=(\lambda_1,\ldots,\lambda_r)\vdash n$. For $T\in \SYT(n)$ and $1\le i\le r$,
  we have
  \[
    \Des_{\lambda,i}(T) = \Des(T^{B_{\lambda,i}}).
  \]
  Consequently,
  \[
    \maj_{\lambda,i}(T) = \maj(T^{B_{\lambda,i}}).
  \]
\end{lem}

\begin{proof}
  By the well-known fact \cite[p.~431]{EC2} that the
  rectification preserves the descent set, we have
  \[
    \Des(T^{B_{\lambda,i}})= \Des(\rect(\st (T_{B_{\lambda,i}}))) = \Des(\st(T_{B_{\lambda,i}})).
  \]
  It is immediate from the definition that
  \( \Des_{\lambda,i}(T) = \Des(\st(T_{B_{\lambda,i}})) \). This shows the first statement.
  The second statement follows from the first.
\end{proof}

\begin{lem}\label{lem:block-product-general}
  Let \( n_1,\dots,n_t \) be positive integers and
  \(n=n_1+\cdots+n_t\). For each \( j\in [t] \), let
  \( \mu^{(j)}\vdash n_j \),
  \(\mathcal A_j \subseteq \SYT(\mu^{(j)}) \), and
\[
  I_j=[n_1+\cdots+n_{j-1}+1,n_1+\cdots+n_j].
\]
Then
\[
\prod_{j=1}^t |\mathcal{A}_j| s_{\mu^{(j)}}
=
\sum_{\mu\vdash n}
\left|\left\{T\in \SYT(\mu): T^{I_j}\in \mathcal A_j\ \text{for all }j\right\}\right|
s_\mu.
\]
\end{lem}

\begin{proof}
By \eqref{eq:schur-fundamental}, we can express the left-hand side of the equation as
\[
\prod_{j=1}^t|\mathcal A_j|\, s_{\mu^{(j)}}
= \prod_{j=1}^t
\sum_{\substack{(P,Q)\in \SYT(n_j)^2\\
P\in \mathcal A_j,\ \sh(P)=\sh(Q)}}
F_{n_j,\Des(Q)}.
\]
By the Robinson--Schensted correspondence
and repeated applications of \eqref{eq:shuffle},
we obtain
\[
\prod_{j=1}^t |\mathcal A_j|\, s_{\mu^{(j)}}
=
\prod_{j=1}^t
\sum_{\substack{w\in S_{n_j}\\
P(w)\in \mathcal A_j}}
F_{n_j,\Des(w)}
= \sum_{w\in A} F_{n,\Des(w)},
\]
where \(A = \{w\in S_n: P\bigl(\st(w_{I_j})\bigr)\in \mathcal A_j\text{ for all }j\} \).
By \Cref{lem:rsk-subword}, $A$ can be rewritten as
\[
A=
\left\{
 w\in S_n : P(w)^{I_j} \in \mathcal A_j \text{ for all }j
\right\}.
\]
Thus, applying the Robinson--Schensted correspondence together with the identity
$\Des(w)=\Des(Q(w))$, we obtain
\begin{align*}
\sum_{w\in A} F_{n,\Des(w)}
&=
\sum_{\mu\vdash n}
\sum_{\substack{w\in S_n\\
\sh(P(w))=\mu\\
P(w)^{I_j}\in \mathcal A_j\ \text{for all }j}}
F_{n,\Des(Q(w))}\\
&=
\sum_{\mu\vdash n}
\sum_{\substack{(P,Q)\in \SYT(\mu)^2\\
P^{I_j}\in \mathcal A_j\ \text{for all }j}}
F_{n,\Des(Q)}\\
&=
\sum_{\mu\vdash n}
\left|
\Bigl\{
P\in \SYT(\mu) : 
P^{I_j}\in \mathcal A_j\ \text{for all }j
\Bigr\}\right|
s_\mu,
\end{align*}
which completes the proof.
\end{proof}

The following lemma will play a key role in what follows.

\begin{lem}\label{lem:prod}
Let $\lambda=(\lambda_1,\dots,\lambda_r)$ be a partition of $n$.  
Then 
\[
\prod_{j=1}^r \ch(L_{\lambda_j})
= \sum_{\mu\vdash n} |\SYT_\lambda(\mu)| s_{\mu} .
\]
\end{lem}

\begin{proof}
By \Cref{thm:KW}, we have
\begin{equation}\label{eq:2}
\prod_{j=1}^r \ch(L_{\lambda_j})
= \sum_{\mu^{(1)},\dots,\mu^{(r)}}
\prod_{j=1}^r |\mathcal A_j^{\mu^{(j)}}|\, s_{\mu^{(j)}},
\end{equation}
where the sum is over all partitions \(\mu^{(1)},\dots,\mu^{(r)}\)
with \(\mu^{(j)}\vdash \lambda_j\), and
\[
  \mathcal A_j^{\mu^{(j)}}=\{U\in \SYT(\mu^{(j)}):\maj(U)\equiv 1 \pmod{\lambda_j}\} .
\]
Applying \Cref{lem:block-product-general} with \(n_j=\lambda_j\),
\(\mathcal A_j=\mathcal A_j^{\mu^{(j)}}\), and \(I_j=B_{\lambda,j}\)
yields
\[
\prod_{j=1}^r |\mathcal A_j^{\mu^{(j)}}|\, s_{\mu^{(j)}}
=
\sum_{\mu\vdash n}
\Bigl|\Bigl\{T\in \SYT(\mu): T^{B_{\lambda,j}}\in \mathcal A_j^{\mu^{(j)}}\ \text{for all }j\Bigr\}\Bigr|\, s_\mu.
\]
By \Cref{lem:rect-des}, the condition
\( T^{B_{\lambda,j}}\in \mathcal A_j^{\mu^{(j)}} \) is equivalent to
\( \maj(T^{B_{\lambda,j}})\equiv 1 \pmod{\lambda_j} \) and
\( \sh(T^{B_{\lambda,j}}) = \mu^{(j)} \). Hence, summing over all
partitions \( \mu^{(j)}\vdash \lambda_j \) for \( j\in [r] \) and
using \eqref{eq:2}, we obtain
\[
\prod_{j=1}^r \ch(L_{\lambda_j})
=
\sum_{\mu\vdash n}
\Bigl|\Bigl\{T\in \SYT(\mu): \maj(T^{B_{\lambda,j}})\equiv 1 \pmod{\lambda_j}\ \text{for all }j\Bigr\}\Bigr|\, s_\mu,
\]
which agrees with the desired identity.
\end{proof}

Since \(\ch(L_\lambda)=\prod_{i\ge 1} \ch(L_{\lambda_i}) \) for a
partition \( \lambda \) with distinct parts, we obtain its Schur
expansion immediately from \Cref{lem:prod}.

\begin{thm}\label{thm:distinct}
Let $\lambda$ be a partition of $n$ with distinct parts.  
Then 
\[
\ch(L_\lambda) = \sum_{\mu\vdash n} |\SYT_\lambda(\mu)| s_\mu.
\]
\end{thm}

When $\lambda=(n)$, we have
\[
\SYT_\lambda(\mu) = \{ T\in \SYT(\mu): \maj(T)\equiv 1 \pmod n\}.
\]
In this case, \Cref{thm:distinct} recovers the classical result of
Kra\'skiewicz--Weyman (\Cref{thm:KW}). We illustrate
\Cref{thm:distinct} with an example.

\begin{exam}\label{exam:distinct-42}
  Consider $n=6$ and $\lambda=(4,2)$. The blocks are
  \(B_{\lambda,1}=\{1,2,3,4\} \) and \( B_{\lambda,2}=\{5,6\} \).
  \Cref{fig:1} shows the standard Young tableaux \( T\in \SYT(6) \) such that
\[
\maj_{\lambda,1}(T)\equiv 1 \pmod 4,
\qquad
\maj_{\lambda,2}(T)\equiv 1 \pmod 2.
\]
Summing the corresponding Schur functions yields
\[
\ch(L_{(4,2)}) = s_{(4,2)}
+s_{(4,1,1)}
+2s_{(3,2,1)}
+2s_{(3,1,1,1)}
+s_{(2,2,2)}
+s_{(2,2,1,1)}
+s_{(2,1,1,1,1)}.
\]
\end{exam}

\begin{figure}
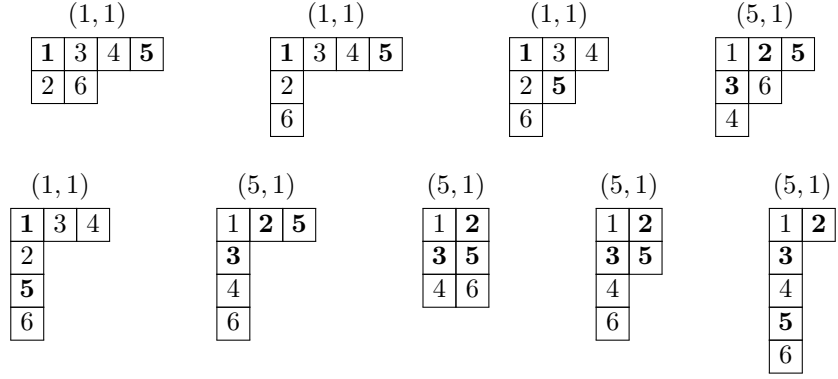

  \centering
\renewcommand{\arraystretch}{1.2}

\begin{tabular}[t]{c@{\qquad\qquad}c@{\qquad\qquad}c@{\qquad\qquad}c}

\begin{tabular}[t]{@{}c@{}}
$(1,1)$\\[0.3ex]
\begin{ytableau}
\mathbf{1} & 3 & 4 & \mathbf{5}\\
2 & 6
\end{ytableau}
\end{tabular}
&
\begin{tabular}[t]{@{}c@{}}
$(1,1)$\\[0.3ex]
\begin{ytableau}
\mathbf{1} & 3 & 4 & \mathbf{5}\\
2\\
6
\end{ytableau}
\end{tabular}
&
\begin{tabular}[t]{@{}c@{}}
$(1,1)$\\[0.3ex]
\begin{ytableau}
\mathbf{1} & 3 & 4\\
2 & \mathbf{5}\\
6
\end{ytableau}
\end{tabular}
&
\begin{tabular}[t]{@{}c@{}}
$(5,1)$\\[0.3ex]
\begin{ytableau}
1 & \mathbf{2} & \mathbf{5}\\
\mathbf{3} & 6\\
4
\end{ytableau}
\end{tabular}

\end{tabular}

\vspace{1em}

\begin{tabular}[t]{c@{\qquad\qquad}c@{\qquad\qquad}c@{\qquad\qquad}c@{\qquad\qquad}c}

\begin{tabular}[t]{@{}c@{}}
$(1,1)$\\[0.3ex]
\begin{ytableau}
\mathbf{1} & 3 & 4\\
2\\
\mathbf{5}\\
6
\end{ytableau}
\end{tabular}
&
\begin{tabular}[t]{@{}c@{}}
$(5,1)$\\[0.3ex]
\begin{ytableau}
1 & \mathbf{2} & \mathbf{5}\\
\mathbf{3}\\
4\\
6
\end{ytableau}
\end{tabular}
&
\begin{tabular}[t]{@{}c@{}}
$(5,1)$\\[0.3ex]
\begin{ytableau}
1 & \mathbf{2}\\
\mathbf{3} & \mathbf{5}\\
4 & 6
\end{ytableau}
\end{tabular}
&
\begin{tabular}[t]{@{}c@{}}
$(5,1)$\\[0.3ex]
\begin{ytableau}
1 & \mathbf{2}\\
\mathbf{3} & \mathbf{5}\\
4\\
6
\end{ytableau}
\end{tabular}
&
\begin{tabular}[t]{@{}c@{}}
$(5,1)$\\[0.3ex]
\begin{ytableau}
1 & \mathbf{2}\\
\mathbf{3}\\
4\\
\mathbf{5}\\
6
\end{ytableau}
\end{tabular}

\end{tabular}

\caption{The tableaux in \( \bigcup_{\mu\vdash 6}\SYT_{(4,2)}(\mu) \).
Above each \( T \), the sequence
\((\maj_{\lambda,1}(T),\maj_{\lambda,2}(T))\) is shown. We highlight
in bold the entries contributing to the block-major indices, i.e.,
those \(d\in\Des(T)\) such that \(d\) and \(d+1\) are contained in the
same block.}
  \label{fig:1}
\end{figure}

We now show that a refined \( (\lambda,\mu) \)-Thrall
subset always exists.

\begin{prop}\label{prop:general-upper-bound}
  For any partitions \( \lambda \) and \( \mu \) of \( n \), we have
\[
c_{\lambda,\mu} \le |\SYT_\lambda(\mu)|.
\]
In particular, a refined \( (\lambda,\mu) \)-Thrall subset always exists.

Moreover, for a given \( \lambda\vdash n \), the equality holds
for all \( \mu\vdash n \) if and only if \( \lambda \) has distinct parts.
\end{prop}

\begin{proof}
  Let \(\lambda=(\lambda_1,\dots,\lambda_r)=(a_1^{k_1},\dots,a_t^{k_t})\) with
  \( a_1>\cdots>a_t\ge1 \) and \( k_1,\dots,k_t\ge1 \). Then
\[
L_\lambda
=
\bigotimes_{j=1}^t \operatorname{Sym}^{k_j}(L_{a_j}).
\]
For each $j\in[t]$, the symmetric power $\operatorname{Sym}^{k_j}(L_{a_j})$ embeds naturally into $L_{a_j}^{\otimes k_j}$ as a $\GL(V)$-submodule. Taking tensor products over $j$, we obtain a natural inclusion
\begin{equation}\label{eq:10}
L_\lambda
=
\bigotimes_{j=1}^t \operatorname{Sym}^{k_j}(L_{a_j})
\hookrightarrow
\bigotimes_{j=1}^t L_{a_j}^{\otimes k_j}
=
\bigotimes_{i=1}^r L_{\lambda_i}.
\end{equation}
For every partition $\mu\vdash n$, it follows that the Schur
coefficient \( c_{\lambda,\mu} \) of $s_\mu$ in $\ch(L_\lambda)$ is at
most that in \(\prod_{i=1}^r \ch(L_{\lambda_i}) \), since multiplicities
of irreducibles are monotone under inclusions of \(\GL(V)\)-modules. By \Cref{lem:prod}, the latter coefficient is equal
to \(|\SYT_\lambda(\mu)| \), and we obtain the inequality.

For the last statement, the ``if'' part follows from
\Cref{thm:distinct}. For the ``only if'' part, suppose that
\( \lambda \) has equal parts, that is, \( k_j\ge2 \) for some \( j \).
Then $\operatorname{Sym}^{k_j}(L_{a_j})$ is a proper
$\GL(V)$-submodule of $L_{a_j}^{\otimes k_j}$, so the inclusion in
\eqref{eq:10} is not surjective. Thus, there exists \( \mu\vdash n \)
such that \( c_{\lambda,\mu}< |\SYT_\lambda(\mu)| \), which completes
the proof.
\end{proof}

Note that every partition \( \lambda \) can be written as the union of
the rectangular blocks \( (i^{m_i(\lambda)}) \) for \( i\ge1 \). The
following proposition shows that to find a refined
\( (\lambda,\mu) \)-Thrall subset, it suffices to find a refined
\( ((i^{m_i(\lambda)}),\nu) \)-Thrall subset for all \( i\ge1 \) and
\( \nu\vdash i \cdot m_i(\lambda) \).

\begin{prop}\label{prop:prod-L-la}
  Let \( \lambda= \lambda^{(1)} \cup \cdots \cup \lambda^{(t)} \) be a
  partition with \( \lambda^{(j)}\vdash n_j \) such that the smallest
  part of \( \lambda^{(i)} \) is greater than the largest part of
  \( \lambda^{(i+1)} \) for all \( i\in [t-1] \). Suppose that
  \( \mathcal{A}(\lambda^{(j)}, \nu) \) is a refined
  \( (\lambda^{(j)}, \nu) \)-Thrall subset for all \( j\in[t] \) and
  \( \nu\vdash n_j \). Let \(n=n_1+\cdots+n_t\) and
\[
  I_j=[n_1+\cdots+n_{j-1}+1,n_1+\cdots+n_j].
\]
Then for any \( \mu\vdash n \), the following is a
refined \( (\lambda,\mu) \)-Thrall subset:
\begin{equation}\label{eq:7}
  \{T\in \SYT_\lambda(\mu): T^{I_j} \in \mathcal A(\lambda^{(j)},\sh(T^{I_j})) \text{ for all }j\}.
\end{equation}
\end{prop}

\begin{proof}
  By the assumption on \( \lambda \), we have
  \(m_i(\lambda)=\sum_{j=1}^t m_i(\lambda^{(j)}) \), with at most one
  nonzero summand. Hence, \eqref{eq:ch(L_la)} gives
\[
\ch(L_\lambda)
=\prod_{i\ge 1} h_{m_i(\lambda)}[\ch(L_i)]
=\prod_{j=1}^t \prod_{i\ge 1} h_{m_i(\lambda^{(j)})}[\ch(L_i)]
=\prod_{j=1}^t \ch(L_{\lambda^{(j)}}).
\]
Since \(\mathcal A(\lambda^{(j)},\nu)\) is a refined
\((\lambda^{(j)},\nu)\)-Thrall subset for every \(\nu\vdash n_j\),
the above can be written as
\[
\ch(L_\lambda)
=\prod_{j=1}^t \sum_{\nu^{(j)}\vdash n_j} |\mathcal A(\lambda^{(j)},\nu^{(j)})|\, s_{\nu^{(j)}}.
\]
For each tuple \((\nu^{(1)},\dots,\nu^{(t)})\) with \(\nu^{(j)}\vdash n_j\), applying \Cref{lem:block-product-general} with \(\mu^{(j)}=\nu^{(j)}\) and \(\mathcal A_j=\mathcal A(\lambda^{(j)},\nu^{(j)})\) yields
\[
\prod_{j=1}^t |\mathcal A(\lambda^{(j)},\nu^{(j)})|\, s_{\nu^{(j)}}
=
\sum_{\mu\vdash n}
\Bigl|\Bigl\{T\in \SYT(\mu):
T^{I_j}\in \mathcal A(\lambda^{(j)},\nu^{(j)})\ \text{for all }j\Bigr\}\Bigr|\, s_\mu.
\]
Summing over \(\nu^{(j)}\vdash n_j\) and using the fact that
\(T^{I_j}\in \mathcal A(\lambda^{(j)},\nu^{(j)})\) forces
\(\nu^{(j)}=\sh(T^{I_j})\), we obtain
\(\ch(L_\lambda) = \sum_{\mu\vdash n} |A(\mu)| s_\mu \),
where
\[
  A(\mu) = \{T\in \SYT(\mu): T^{I_j}\in \mathcal
  A(\lambda^{(j)},\sh(T^{I_j}))\ \text{for all }j\}.
\]
 Thus, it
suffices to show that \( A(\mu) \) is equal to the set in
\eqref{eq:7}, which we denote by \( B(\mu) \). Since
\( B(\mu) \subseteq A(\mu) \), it remains to show that
\( A(\mu) \subseteq B(\mu) \), or equivalently, if \( T\in A(\mu) \),
then \( T\in \SYT_\lambda(\mu) \).

Let \( T\in A(\mu) \).
By the assumption on \( \lambda \), we have
\[
\lambda=\bigl(\lambda^{(1)}_1,\dots,\lambda^{(1)}_{r_1},\lambda^{(2)}_1,\dots,\lambda^{(2)}_{r_2},\dots,\lambda^{(t)}_1,\dots,\lambda^{(t)}_{r_t}\bigr),
\]
where \(r_j= \ell(\lambda^{(j)})\). Let \(R_j=r_1+\cdots+r_j\).
To prove \( T\in \SYT_\lambda(\mu) \), we must show that
for all \(j\in[t]\) and \(k\in[r_j]\),
\begin{equation}\label{eq:8}
  \maj_{\lambda,R_{j-1}+k}(T)\equiv 1 \pmod{\lambda_{R_{j-1}+k}}.
\end{equation}

We have \(\lambda_{R_{j-1}+k}=\lambda^{(j)}_k\), and
\(B_{\lambda,R_{j-1}+k}=\{(n_1+\cdots+n_{j-1})+a: a\in
B_{\lambda^{(j)},k}\} \) is contained in \(I_j\). Observe that for any
sets \( I \) and \( J \) of consecutive integers with
\( I\subseteq J\subseteq[n] \), we have \( T^I = (T^J)^{I'} \), where
\( I'=\{i-\min(J)+1:i\in I\} \), since \( T^I = \rect(\st(T_I)) \) and
the rectification does not depend on the sequence of jeu de taquin
slides. This shows that
\(T^{B_{\lambda,R_{j-1}+k}} = (T^{I_j})^{B_{\lambda^{(j)},k}} \).
Then by \Cref{lem:rect-des}, we have
\begin{equation}\label{eq:9}
\maj_{\lambda,R_{j-1}+k}(T)
=\maj\bigl(T^{B_{\lambda,R_{j-1}+k}}\bigr)
=\maj\bigl((T^{I_j})^{B_{\lambda^{(j)},k}}\bigr)
=\maj_{\lambda^{(j)},k}(T^{I_j}).
\end{equation}

On the other hand, since \( T\in A(\mu) \), letting
\(\nu^{(j)}=\sh(T^{I_j})\), we have
\(T^{I_j}\in \mathcal A(\lambda^{(j)},\nu^{(j)})\), which is a refined
\( (\lambda^{(j)},\nu^{(j)}) \)-Thrall subset. Thus, we have
\( T^{I_j}\in \SYT_{\lambda^{(j)}}(\nu^{(j)}) \), that is,
\(\maj_{\lambda^{(j)},k}(T^{I_j})\equiv 1 \pmod{\lambda^{(j)}_k}\) for
all \( k\in [r_j] \). Therefore, \eqref{eq:9} together with the fact
that \(\lambda_{R_{j-1}+k}=\lambda^{(j)}_k\) implies \eqref{eq:8},
which completes the proof.
\end{proof}

As an application of \Cref{prop:prod-L-la}, we find
\( (\lambda,\mu) \)-Thrall subsets when \( \lambda \) is a hook that
is not a one-column partition.

\begin{thm}\label{thm:hook-general}
Let $\lambda=(n-k,1^k)$ be a hook partition of $n$ with $0\le k\le n-2$. Then
\[
\ch(L_{(n-k,1^k)})
=
\sum_{\mu\vdash n} |\{T\in \SYT_{(n-k,1^k)}(\mu): \Des(T)\subseteq [n-k]\}| s_\mu.
\]
\end{thm}

\begin{proof}
Observe that \eqref{eq:one-col} is equivalent to
\[
  \ch(L_{(1^n)})= \sum_{\mu\vdash n} |\{T\in \SYT_{(1^n)}(\mu):\Des(T)=\emptyset\}| s_\mu.
\]
Hence, we obtain the result by \Cref{thm:KW} and \Cref{prop:prod-L-la}.
\end{proof}

\section{The two-row case and partitions with multiplicity at most two}\label{sec:equal-parts}

In this section, we solve Thrall's problem when \( \lambda=(n,n) \) by
finding a refined \( (\lambda,\mu) \)-Thrall subset. Recall from
\Cref{prop:general-upper-bound} that
\(c_{\lambda,\mu} \le |\SYT_\lambda(\mu)| \), where the equality holds
for all \( \mu \) if and only if \( \lambda \) has distinct parts.
When \( \lambda \) has equal parts, we need to find a condition on the
elements in \( \SYT_\lambda(\mu) \) such that \( c_{\lambda,\mu} \) is
equal to the number of tableaux satisfying this condition.

\subsection{The main result}
\label{sec:main-result}

In order to state our main result, we need some definitions. First, we
introduce a total order \( \le \) on the set of all standard Young
tableaux of a partition shape as follows.
The purpose of the order is to list all standard Young tableaux,
and the results in this section are independent of the choice of the order.

\begin{defn}\label{def:3}
  For two standard Young tableaux \( P \) and \( Q \) of partition
  shape, we define \( P\le Q \) if one of the following conditions
  holds:
\begin{itemize}
\item The size of \( P \) is less than the size of \( Q \).
\item The sizes of \( P \) and \( Q \) are equal, and the reading word
  of \( P \) is lexicographically at most that of \( Q \).
\end{itemize}
We write \( P<Q \) to mean \( P\le Q \) and \( P\ne Q \).
\end{defn}

For a nonnegative integer \( n \), we define
\begin{align*}
  \SYT^{=}(2n) &= \{T\in \SYT(2n): T_{[n]} = T^{[n+1,2n]}\}.
\end{align*}
The following lemma is the key ingredient of our approach.
We postpone its proof to \Cref{subsec:nn-bijections}.

\begin{lem}\label{lem:xi}
  There is a bijection
  \[
    \xi: \SYT^=(2n) \to \bigsqcup_{\substack{\lambda\vdash n \\ \mu\vdash 2n}}
    \SYT(\lambda) \times \YDT(\lambda^\square,\mu)
  \]
  such that if \( \xi(T) = (U,D) \), then \( U = T_{[n]} \) and the
  weight of \( D \) is equal to \( \sh(T) \).
\end{lem}

The bijection \( \xi \) in \Cref{lem:xi} allows us to define a spin
statistic on \( \SYT^=(2n) \).

\begin{defn}\label{def:4}
  For \( T\in \SYT^=(2n) \) with \( \xi(T) = (U,D) \), we define
  the \emph{spin} of \( T \) by
  \[
    \spin(T):=\spin(D).
  \]
\end{defn}

Now we state our main result whose proof is given in \Cref{sec:proof-main}.

\begin{thm}\label{thm:nn-final-single-sum}
We have
\[
\ch(L_{(n,n)})
=
\sum_{\mu\vdash 2n}
|\SYT_{(n,n)}^{<}(\mu)\sqcup \SYT_{(n,n)}^{\spin}(\mu)|
s_\mu,
\]
where
\begin{align*}
 \SYT_{(n,n)}^{<}(\mu) &= \{T\in \SYT_{(n,n)}(\mu): T_{[n]} < T^{[n+1,2n]}\},\\
    \SYT_{(n,n)}^{\spin}(\mu) &= 
    \left\{T\in \SYT_{(n,n)}(\mu): T_{[n]}=T^{[n+1,2n]},\, \spin(T)\equiv n \pmod 2
\right\}.
  \end{align*}
In other words,
\( \SYT_{(n,n)}^{<}(\mu)\sqcup \SYT_{(n,n)}^{\spin}(\mu) \) is a refined
\( ((n,n),\mu) \)-Thrall subset.
\end{thm}

\begin{figure}
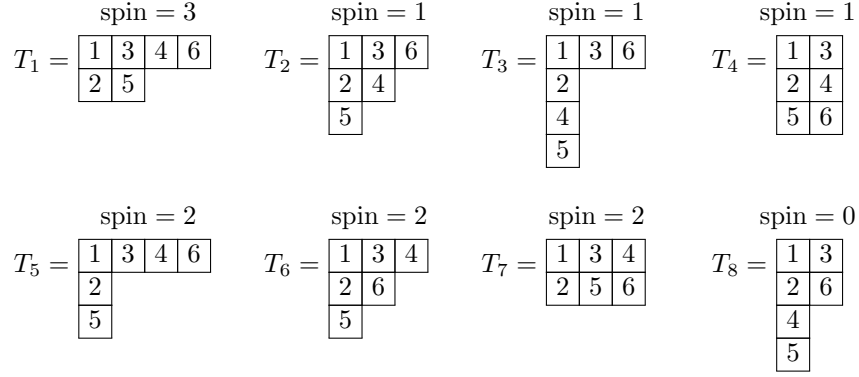

\centering
\renewcommand{\arraystretch}{1.2}

\begin{tabular}[t]{c@{\qquad}c@{\qquad}c@{\qquad}c}

\begin{tabular}[t]{@{}c@{}}
\hspace{2.7em}$\spin=3$\\[0.3ex]
$T_1 =$~\begin{ytableau}
1 & 3 & 4 & 6\\
2 & 5
\end{ytableau}
\end{tabular}
&
\begin{tabular}[t]{@{}c@{}}
\hspace{2.5em}$\spin=1$\\[0.3ex]
$T_2 =$~\begin{ytableau}
1 & 3 & 6 \\
2 & 4\\
5
\end{ytableau}
\end{tabular}
&
\begin{tabular}[t]{@{}c@{}}
\hspace{2.5em}$\spin=1$\\[0.3ex]
$T_3 =$~\begin{ytableau}
1 & 3 & 6\\
2 \\
4 \\
5
\end{ytableau}
\end{tabular}
&
\begin{tabular}[t]{@{}c@{}}
\hspace{2.3em}$\spin=1$\\[0.3ex]
$T_4 =$~\begin{ytableau}
1 & 3 \\
2 & 4 \\
5 & 6
\end{ytableau}
\end{tabular}

\end{tabular}

\vspace{1em}

\begin{tabular}[t]{c@{\qquad}c@{\qquad}c@{\qquad}c}

\begin{tabular}[t]{@{}c@{}}
\hspace{2.7em}$\spin=2$\\[0.3ex]
$T_5 =$~\begin{ytableau}
1 & 3 & 4 & 6\\
2 \\ 
5
\end{ytableau}
\end{tabular}
&
\begin{tabular}[t]{@{}c@{}}
\hspace{2.5em}$\spin=2$\\[0.3ex]
$T_6 =$~\begin{ytableau}
1 & 3 & 4 \\
2 & 6 \\
5
\end{ytableau}
\end{tabular}
&
\begin{tabular}[t]{@{}c@{}}
\hspace{2.5em}$\spin=2$\\[0.3ex]
$T_7 =$~\begin{ytableau}
1 & 3 & 4\\
2 & 5 & 6
\end{ytableau}
\end{tabular}
&
\begin{tabular}[t]{@{}c@{}}
\hspace{2.3em}$\spin=0$\\[0.3ex]
$T_8 =$~\begin{ytableau}
1 & 3 \\
2 & 6 \\
4 \\
5
\end{ytableau}
\end{tabular}

\end{tabular}
\caption{The tableaux in
\(\bigcup_{\mu\vdash 6}\SYT_{(3,3)}(\mu)\)
satisfying \(T_{[3]} = T^{[4,6]}\), together with their spin values.}
\label{fig:nn-n3-example}
\end{figure}

\begin{exam}\label{ex:33}
There are eight tableaux \( T \) in
\( \bigcup_{\mu\vdash 6}\SYT_{(3,3)}(\mu) \)
satisfying \(T_{[3]} = T^{[4,6]} \),
as shown in \Cref{fig:nn-n3-example}, together with their spin values.
Among these, \(T_1,T_2,T_3\), and \(T_4\) have odd spin, and hence satisfy
\(\spin(T)\equiv 3 \pmod 2\). 
These are precisely the tableaux that contribute to the spin part
in \Cref{thm:nn-final-single-sum}, while the remaining four are excluded.
Since \( \SYT_{(3,3)}^{<}(\mu)=\varnothing \) for all \( \mu \vdash 6 \),
by \Cref{thm:nn-final-single-sum}, we obtain
\[
\ch(L_{(3,3)})
= s_{(4,2)}+s_{(3,2,1)}+s_{(3,1,1,1)}+s_{(2,2,2)}.
\]
\end{exam}

By \Cref{prop:prod-L-la}, we can find a refined
\( (\lambda,\mu) \)-Thrall subset using refined
\( ((i^{m_i(\lambda)}),\nu) \)-Thrall subsets for \( i\ge1 \) and
\( \nu\vdash i \cdot m_i(\lambda) \). Hence, applying this to combine the results of 
\Cref{thm:KW,thm:nn-final-single-sum}, \eqref{eq:one-col}, and \eqref{eq:two-col},
which are the one-row case, the two-row case, the one-column case, and
the two-column case, respectively, we obtain the following theorem. This solves
Thrall's problem for the partitions \( \lambda \) in which every
part greater than \( 2 \) may appear at most twice.

\begin{thm}\label{thm:at-most-two-with-ones}
  Let \(\lambda=(a_1^{k_1},\dots,a_{t-1}^{k_{t-1}},a_{t}^{k_{t}})\) be a partition of
  \( n \) such that \(a_1>\cdots>a_{t-2}>a_{t-1}=2>a_{t}=1\), \(k_j\in \{1,2\}\) for all
  \( j\in [t-2] \), and \( k_{t-1},k_t\ge0 \). 
  For each \(j\in[t]\), set
\[
\lambda^{(j)}=(a_j^{k_j}),\qquad
N_j=k_1a_1+\cdots+k_ja_j,\qquad
I_j=[N_{j-1}+1,N_j],
\]
where \( N_0=0 \). For \( j\in [t] \) and \(\nu\vdash |\lambda^{(j)}|\), define
\[
\mathcal A(\lambda^{(j)},\nu)=
\begin{cases}
\SYT_{(a_j)}(\nu),
& \text{if } j\le t-2 \text{ and } k_j=1,\\[4pt]
\SYT_{(a_j,a_j)}^{<}(\nu)\sqcup
\SYT_{(a_j,a_j)}^{\spin}(\nu),
& \text{if } j\le t-2 \text{ and } k_j=2,\\[4pt]
\{U\in\SYT_{(2^{k_{t-1}})}(\nu):\Des(U)=\maxDes(\nu)\},
& \text{if } j=t-1,\\
\{U\in\SYT(\nu):\Des(U)=\emptyset\},
& \text{if } j=t.
\end{cases}
\]
Then for any \(\mu\vdash n\),
\[
\left\{
T\in \SYT_\lambda(\mu): T^{I_j}\in
\mathcal A(\lambda^{(j)},\sh(T^{I_j})) \text{ for all }j
\right\}
\]
is a refined \((\lambda,\mu)\)-Thrall subset. In other words, we have
\[
\ch(L_\lambda)
=
\sum_{\mu\vdash n}
\left|
\left\{
T\in \SYT_\lambda(\mu): T^{I_j}\in
\mathcal A(\lambda^{(j)},\sh(T^{I_j})) \text{ for all }j
\right\}
\right|
s_\mu.
\]
\end{thm}


\subsection{Proof of \Cref{thm:nn-final-single-sum}}
\label{sec:proof-main}

In this subsection, we prove \Cref{thm:nn-final-single-sum} assuming
\Cref{lem:xi}. For simplicity, we define
\begin{equation}\label{eq:a_la}
  a_\lambda = |\{T\in\SYT(\lambda): \maj(T)\equiv 1 \pmod n\}|
  = |\SYT_{(n)}(\lambda)|.
\end{equation}
Then, by \Cref{thm:KW}, we have
\begin{equation}\label{eq:3}
  \ch(L_n)=\sum_{\lambda \vdash n} a_{\lambda} s_{\lambda}.
\end{equation}

The following lemma is the starting point of our
investigation of \( \ch(L_{(n,n)}) \).

\begin{lem}\label{lem:nn}
We have
\[
\ch(L_{(n,n)})
= \sum_{\lambda \vdash n} a_{\lambda}h_2[s_{\lambda}]
+ \frac{1}{2}\left( \ch(L_n)^2 -\sum_{\lambda \vdash n}a_{\lambda}s_{\lambda}^2 \right).
\]
\end{lem}

\begin{proof}
Since \( h_2= (p_1^2 + p_2)/2 \), we have
\[
  \ch(L_{(n,n)})
=h_2[\ch(L_n)]
=\frac{1}{2}\left( \ch(L_n)^2 + p_2[\ch(L_n)] \right)
=\frac{1}{2}\left( \ch(L_n)^2 + \sum_{\lambda \vdash n}a_{\lambda} p_2[s_{\lambda}]\right).
\]
Since \( p_2= 2h_2-p_1^2 \), we have
\[
 \sum_{\lambda \vdash n}a_{\lambda} p_2[s_{\lambda}]
=\sum_{\lambda \vdash n}a_{\lambda}(2h_2[s_{\lambda}]-s_{\lambda}^2).
\]
Combining the above two identities gives the lemma.
\end{proof}

By \Cref{lem:nn}, to prove \Cref{thm:nn-final-single-sum},
it suffices to prove the following two propositions.

\begin{prop}\label{prop:SYT-sp}
 We have
 \[
\sum_{\lambda\vdash n} a_\lambda h_2[s_\lambda]
=
\sum_{\mu\vdash 2n}
\left| \SYT_{(n,n)}^{\spin}(\mu) \right|
s_\mu.
\]
\end{prop}

\begin{prop}\label{prop:SYT^<}
  We have
\[
\frac{1}{2} \left( \ch(L_n)^2 -\sum_{\lambda \vdash n}a_{\lambda}s_{\lambda}^2 \right)
=\sum_{\mu\vdash 2n} \left| \SYT_{(n,n)}^{<}(\mu)\right|
s_\mu.
\]
\end{prop}

The rest of this subsection is devoted to proving the above two
propositions. We first prove \Cref{prop:SYT-sp}.

\begin{proof}[Proof of \Cref{prop:SYT-sp}]
 By \eqref{eq:a_la} and \Cref{thm:CL}, we have
\[
\sum_{\lambda\vdash n} a_\lambda h_2[s_\lambda]
= \sum_{\mu\vdash 2n} \left| \bigsqcup_{\lambda\vdash n} \SYT_{(n)}(\lambda) \times Y(\lambda,\mu)  \right| s_\mu,
\]
where
\[
  Y(\lambda,\mu)= \{D\in\YDT(\lambda^\square,\mu): \spin(D) \equiv n \pmod{2}\}.
\]
Thus, it suffices to find a bijection
\[
 \eta: \SYT_{(n,n)}^{\spin}(\mu) \to \bigsqcup_{\lambda\vdash n} \SYT_{(n)}(\lambda) \times Y(\lambda,\mu) .
\]
The restriction of the map \( \xi \) in \Cref{lem:xi} to
\( \SYT_{(n,n)}^{\spin}(\mu) \) induces such a bijection.
\end{proof}

To prove \Cref{prop:SYT^<}, we need some definitions and lemmas.
For \( \mu\vdash 2n \), we define
\begin{align*}
  \SYT_{(n,n)}^{\ne}(\mu) &= \{T\in \SYT_{(n,n)}(\mu): T_{[n]} \ne T^{[n+1,2n]}\},\\
  \SYT_{(n,n)}^{=}(\mu) &= \{T\in \SYT_{(n,n)}(\mu): T_{[n]} = T^{[n+1,2n]}\}.
\end{align*}

\begin{lem}\label{lem:nn-reduction}
  We have
\[
\ch(L_n)^2 -\sum_{\lambda \vdash n}a_{\lambda}s_{\lambda}^2
=\sum_{\mu\vdash 2n}
\left| \SYT_{(n,n)}^{\ne}(\mu) \right|
s_\mu.
\]
\end{lem}

\begin{proof}
  By \Cref{lem:prod} with \( \lambda=(n,n) \), we have
\[
\ch(L_n)^2 = \sum_{\mu\vdash 2n} |\SYT_{(n,n)}(\mu)| s_{\mu} .
\]
Thus, it suffices to show that
\begin{equation}\label{eq:6}
\sum_{\lambda \vdash n}a_\lambda s_\lambda^2
= \sum_{\mu\vdash 2n} |\SYT_{(n,n)}^{=}(\mu)| s_{\mu} .
\end{equation}

Fix \(\mu\vdash 2n\). Let \(c_{\alpha,\beta}^\mu\) denote the
Littlewood--Richardson coefficient, that is, the coefficient of
\(s_\mu\) in \(s_{\alpha}s_{\beta}\). Comparing the coefficients of
\(s_\mu\) on both sides, we see that \eqref{eq:6} is equivalent to
\begin{equation}\label{eq:1}
  \sum_{\lambda\vdash n} a_\lambda c_{\lambda,\lambda}^\mu
  = |\SYT_{(n,n)}^{=}(\mu)|.
\end{equation}
Note that \(c_{\lambda,\lambda}^\mu\) is also equal to the coefficient
of \( s_\lambda \) in \( s_{\mu/\lambda} \). Hence, by the
Littlewood--Richardson rule, we have
\(c_{\lambda,\lambda}^\mu = |\LR(\mu/\lambda,\lambda)|\). Then, by
\Cref{lem:lr-rectification}, we obtain
\(c_{\lambda,\lambda}^\mu = |\SYT^U(\mu/\lambda)|\) for any
\( U\in \SYT(\lambda) \). Since \(a_\lambda\) is the number of
tableaux \(U\in \SYT(\lambda)\) with \(\maj(U)\equiv 1 \pmod n\), the
left-hand side of \eqref{eq:1} is equal to the number of pairs
\((U,S)\) such that
\begin{equation}\label{eqn:conditions}
U\in \SYT(\lambda),\qquad \maj(U)\equiv 1 \pmod n,\qquad S\in \SYT^U(\mu/\lambda),
\end{equation}
for some \(\lambda\vdash n\). Thus, to establish \eqref{eq:1}, it
suffices to find a bijection between \(\SYT_{(n,n)}^{=}(\mu)\) and the
set of pairs \((U,S)\) satisfying \eqref{eqn:conditions}.

Given \(T\in \SYT_{(n,n)}^=(\mu)\), let \( U=T_{[n]} \) and
\( S=\st(T_{[n+1,2n]}) \). Then
\(\maj(U)=\maj_{(n,n),1}(T) \equiv 1 \pmod n\), and since
\( T^{[n+1,2n]} = T_{[n]} \), we have \(S\in \SYT^U(\mu/\lambda)\),
where \(\lambda=\sh(U)\). Thus the pair \( (U,S) \) satisfies
\eqref{eqn:conditions}.

Conversely, for any pair \( (U,S) \) satisfying
\eqref{eqn:conditions}, there is a unique \(T\in \SYT(\mu)\) such that
\( U=T_{[n]} \) and \( S=\st(T_{[n+1,2n]}) \). Then
\[
  T^{[n+1,2n]} = \rect(\st(T_{[n+1,2n]})) = \rect(S) = U = T_{[n]}.
\]
 By
\Cref{lem:rect-des}, we have
\(\maj_{(n,n),1}(T)=\maj(U) \equiv 1 \pmod n \) and
\(\maj_{(n,n),2}(T)=\maj(U) \equiv 1 \pmod n \). Thus,
\(T\in \SYT_{(n,n)}^=(\mu) \), and the map \( T\mapsto(U,S) \) is
a desired bijection, completing the proof.
\end{proof}

\begin{lem}\label{lem:two-row-swap}
Let \(\mu\vdash 2n\). There is a fixed-point-free involution
\[
\iota:\SYT_{(n,n)}^{\ne}(\mu)\to \SYT_{(n,n)}^{\ne}(\mu)
\]
such that \(\iota(T)_{[n]}=T^{[n+1,2n]} \) and
\( \iota(T)^{[n+1,2n]}=T_{[n]} \) for all \(T\in \SYT_{(n,n)}^{\ne}(\mu)\).
\end{lem}

\begin{proof}
  Consider \(T\in \SYT_{(n,n)}^{\ne}(\mu)\). Let \( U=T_{[n]} \),
  \( S=\st(T_{[n+1,2n]}) \), and \(X(U,S)=(S',U') \), where \( X \) is
  the tableau switching map. By the properties of tableau switching,
  we have \( S'=\rect(S)=T^{[n+1,2n]} \) and \(\rect(U')=U \). 
  We define \(\iota(T)\) to be the standard Young tableau of shape \(\mu\)
  obtained by placing \(S'\) in its shape and \(U'\), with all entries
  increased by \(n\), in the complementary skew shape.

  By construction, we have
\[
\iota(T)_{[n]}=S'=T^{[n+1,2n]},
\qquad
\iota(T)^{[n+1,2n]}=\rect(U')=U=T_{[n]}.
\]
By \Cref{lem:rect-des}, we have
\[
\maj_{(n,n),1}(\iota(T))= \maj(T^{[n+1,2n]})\equiv 1 \pmod n
\]
and
\[
\maj_{(n,n),2}(\iota(T))=\maj(\iota(T)^{[n+1,2n]})=\maj(T_{[n]})\equiv 1 \pmod n.
\]
Thus \(\iota(T)\in \SYT_{(n,n)}^{\ne}(\mu)\). Since tableau switching
is an involution, we have \( X(S',U')=(U,S) \). Therefore, the same
construction applied to \(\iota(T)\) recovers \(T\), so \(\iota\) is
an involution. Since \(S'\neq U\), there are no fixed points of
\( \iota \), which completes the proof.
\end{proof}

Finally, we are ready to prove \Cref{prop:SYT^<}.

\begin{proof}[Proof of \Cref{prop:SYT^<}]
  By \Cref{lem:nn-reduction}, it suffices to show that
  \( \frac{1}{2}|\SYT_{(n,n)}^{\ne}(\mu)|= |\SYT_{(n,n)}^{<}(\mu)| \)
  for every \(\mu\vdash 2n\). The fixed-point-free involution
  \( \iota \) in \Cref{lem:two-row-swap} partitions
  \( \SYT_{(n,n)}^{\ne}(\mu) \) into blocks \( \{T,\iota(T)\} \) of
  size \( 2 \). Since \(\iota(T)_{[n]}=T^{[n+1,2n]} \) and
  \( \iota(T)^{[n+1,2n]}=T_{[n]} \), exactly one of \( T \) and
  \( \iota(T) \) is contained in \( \SYT_{(n,n)}^{<}(\mu) \). This
  implies \( \frac{1}{2}|\SYT_{(n,n)}^{\ne}(\mu)|= |\SYT_{(n,n)}^{<}(\mu)| \),
  as desired.
\end{proof}


\subsection{Proof of \Cref{lem:xi}}\label{subsec:nn-bijections}

In this subsection, we prove \Cref{lem:xi} by constructing the
bijection
\begin{equation}\label{eq:5}
\xi: \SYT^=(2n) \to \bigsqcup_{\substack{\lambda\vdash n \\ \mu\vdash 2n}}
    \SYT(\lambda) \times \YDT(\lambda^\square,\mu)
\end{equation} 
 using three bijections introduced by van Leeuwen in
\cite{vanLeeuwen2000, vanLeeuwen2001}.

For partitions \(\lambda\) and \(\mu\), let \(\lambda * \mu\) denote
the disconnected skew shape obtained by placing the Young diagram of
\(\lambda\) strictly to the left and below that of \(\mu\).

Recall that \( \LR(\mu/\lambda,\nu) \) is the set of Littlewood--Richardson
tableaux of shape \(\mu/\lambda\) and weight \(\nu\), and
\(\YDT(\lambda^\square,\mu)\) is the set of Yamanouchi domino
tableaux of shape
\(\lambda^\square=(2\lambda_1,2\lambda_1,2\lambda_2,2\lambda_2,\dots)\)
and weight \(\mu\). Recall also that for \(U\in \SYT(\lambda)\),
\[
  \SYT^{U}(\mu/\lambda) = \{ S\in \SYT(\mu/\lambda):\rect(S)=U \}.
\]

Consider \( T\in \SYT^=(2n) \), and let
\( U=T_{[n]}\in \SYT(\lambda) \) and
\( S=\st(T_{[n+1,2n]})\in \SYT(\mu/\lambda) \), where
\( \mu\vdash 2n \) and \( \lambda\vdash n \). Then \( T \) is
determined by \( U \) and \( S \). Since \( T_{[n]}=T^{[n+1,2n]} \),
we have \( S\in \SYT^{U}(\mu/\lambda) \).
Hence, the map \( \vartheta(T)= (U,S) \) gives a bijection
\begin{equation}\label{eq:vartheta}
 \vartheta: \SYT^=(2n) \to \bigsqcup_{\substack{\lambda\vdash n\\ \mu\vdash 2n}} 
\bigsqcup_{U\in \SYT(\lambda)} \{U\} \times \SYT^{U}(\mu/\lambda) .
\end{equation}
Thus, to find the map \( \xi \), it suffices to construct the
following sequence of bijections for fixed \( \lambda \), \( \mu \),
and \( U\in \SYT(\lambda) \):
\begin{equation}\label{eq:three-maps}
\SYT^{U}(\mu/\lambda)
\xrightarrow{\;\Psi_U\;}
\LR(\mu/\lambda,\lambda)
\xrightarrow{\;\Omega\;}
\LR(\lambda*\lambda,\mu)
\xrightarrow{\;\Phi_0\;}
\YDT(\lambda^\square,\mu).
\end{equation}
See \Cref{def:xi} for the precise description of the bijection \( \xi \).

For the rest of this subsection, we describe the bijections \(\Psi_U\),
\(\Omega\), and \(\Phi_0\).

\subsubsection{The maps \texorpdfstring{\(\Psi_U\)}{PsiU} and \texorpdfstring{\(\Omega\)}{Omega}}\label{subsec:maps}

For a partition \(\lambda\), let \(Q_\lambda\) denote the standard
Young tableau of shape \(\lambda\) obtained by filling the cells row
by row with \(1,2,\dots,|\lambda| \) in this order. Recall that
\(1_\lambda\) denotes the semistandard Young tableau of shape
\(\lambda\) and weight \(\lambda\) whose entries in row \(i\) are all
equal to \(i\). We will use the tableau switching map \( X \) in
\Cref{def:tableau-switching-explicit}.

\begin{defn}\label{def:Psi_U}
  Fix \( \mu\vdash 2n \), \( \lambda\vdash n \), and
  \(U\in \SYT(\lambda)\). The map
\[
\Psi_U: \SYT^{U}(\mu/\lambda) \to \LR(\mu/\lambda,\lambda)
\]
is defined as follows. For \(S \in \SYT^{U}(\mu/\lambda)\), let
\(X(Q_\lambda,S)=(A,B) \) and \(X(1_\lambda,B)=(C,D) \). Then we
define \( \Psi_U(S)=D \).
\end{defn}

Observe that in \Cref{def:Psi_U}, since
\( S \in \SYT^{U}(\mu/\lambda) \), by \Cref{prop:ts} we have
\( A=\rect(S)=U \) and \( \rect(B)=Q_\lambda \). Applying
\Cref{prop:ts} again to \( X(1_\lambda,B)=(C,D) \) gives
\( \wt(D) = \wt(1_\lambda)=\lambda \) and \( C=\rect(B)=Q_\lambda \), which implies
\( \sh(D) = \mu/\lambda \). By \cite[Corollary~2.5.2]{vanLeeuwen2001},
\(D\) is Littlewood--Richardson, and the map \(\Psi_U\) is a bijection.

\begin{exam}\label{exam:psi-running}
Let \( \mu=(5,3,2) \), \( \lambda=(3,2) \), and
\[
U=
\begin{ytableau}
1 & 2 & 4 \\
3 & 5
\end{ytableau}
\qquad \text{and} \qquad
S=
\begin{ytableau}
\none & \none & \none & 2 & 4 \\
\none & \none & 5 \\
1 & 3
\end{ytableau}
~\in \SYT^U(\mu/\lambda).
\]
Applying tableau switching (see \Cref{fig:2}), we obtain
\(X(Q_{\lambda},S)=(A,B) \),
where \( A=U \) and
\[
B=
\begin{ytableau}
\none & \none & \none & 2 & 3 \\
\none & \none & 1 \\
4 & 5
\end{ytableau}\,.
\]
Applying tableau switching again, we obtain
\(X(1_{\lambda},B)=(C,D) \), where \( C=Q_{\lambda} \) and
\[
D=\Psi_U(S)=
\begin{ytableau}
\none & \none & \none & 1 & 1 \\
\none & \none & 2 \\
1 & 2
\end{ytableau}
~\in \LR(\mu/\lambda,\lambda).
\]
\end{exam}

We now describe the bijection \(\Omega\).

\begin{defn}\label{def:Omega}
  Fix \( \mu\vdash 2n \) and \( \lambda\vdash n \). The map
\[
\Omega:\LR(\mu/\lambda,\lambda)\to \LR(\lambda*\lambda,\mu)
\]
is defined as follows. Let \(L\in \LR(\mu/\lambda,\lambda)\). For each
\(l,k\ge1\), suppose that \( L \) contains \( a_{l,k} \) entries equal
to \( k \) in row \( l \). Let \( M \) be the unique tableau in
\( \SSYT(\lambda) \) that contains \(a_{l,k}\) entries equal to \(l\)
in row \( k \) for each \( k,l\ge1 \). We then define
\(\Omega(L)\) to be the tableau of shape
\(\lambda*\lambda\) whose lower-left component is \(M\) and whose
upper-right component is \(1_\lambda\).
\end{defn}

By \cite[Proposition~1.4.5]{vanLeeuwen2001}, the map
\(\Omega\) is a bijection.

\begin{exam}\label{exam:omega-running}
Let \( \mu=(5,3,2) \), \( \lambda=(3,2) \), and
\[
L=
\begin{ytableau}
\none & \none & \none & 1 & 1 \\
\none & \none & 2 \\
1 & 2
\end{ytableau}
~\in \LR(\mu/\lambda,\lambda).
\]
Then the integers \( a_{l,k} \) and the tableau \(M\) in
\Cref{def:Omega} are given by
\[
(a_{l,k})=
\begin{pmatrix}
2 & 0 \\
0 & 1 \\
1 & 1
\end{pmatrix}
\qquad \text{and} \qquad
M=
\begin{ytableau}
1 & 1 & 3 \\
2 & 3
\end{ytableau}\,.
\]
Therefore, we have
\[
\Omega(L)=
\begin{ytableau}
\none & \none & \none & 1 & 1 & 1 \\
\none & \none & \none & 2 & 2 \\
1 & 1 & 3 \\
2 & 3
\end{ytableau}
~\in \LR(\lambda*\lambda,\mu).
\]
\end{exam}

\subsubsection{The map \texorpdfstring{\(\Phi_0\)}{Phi0}}

In what follows, we review the following bijection due to van Leeuwen
\cite[Theorem~2.2.6]{vanLeeuwen2000}:
\[
\Phi_0:\LR(\lambda*\lambda,\mu)\to \YDT(\lambda^\square,\mu).
\]

For \( k\ge1 \), let \( \delta_k = (k-1,k-2,\dots,1) \) denote the
staircase partition with \( k-1 \) parts.

The following is a specialization of van Leeuwen's map
\cite[Definition~2.2.1]{vanLeeuwen1999}. The original map is described
using \( 2 \)-quotients, but in our case, there is a simple
description using explicit coordinates.

\begin{defn}\label{def:8}
  Let \( \mu\vdash 2n \), \( \lambda\vdash n \), and
  \( k=\lambda_1+\ell(\lambda) \). Consider
  \(M\in \LR(\lambda*\lambda,\mu) \). Let \( M_1 \) and \( M_2 \) be
  the lower-left component and the upper-right component of \( M \),
  respectively. To define the domino tableau \( d(M) \), we start with
  the staircase partition \( \delta_k=(k-1,k-2,\dots,1) \) as the inner shape. 
  For each cell \( (i,j) \) of \( M_1 \) with entry \( a \), place the vertical
  domino \( \{(k+2i-j-1,j),(k+2i-j,j)\} \) with label \( a \). For
  each cell \( (i,j) \) of \( M_2 \) with entry \( a \), place the
  horizontal domino \( \{(i,k+2j-i-1), (i,k+2j-i)\} \) with label
  \( a \). Then \( d(M) \) is defined to be the resulting domino
  tableau.
\end{defn}

For an illustration, see \Cref{fig:sigma-example}. Note that the domino tableau \( d(M) \) has a skew shape obtained by removing the staircase partition \(\delta_k\). To move it one step toward a straight shape, we use the following explicit form of van Leeuwen's construction of chains and the operation of moving open chains for domino tableaux, specialized to the case where the inner shape is the staircase partition \(\delta_k\) \cite[Sections~4.2–4.5]{vanLeeuwen1999}.

\begin{figure} [t]
\centering
\(
M=\begin{ytableau}
\none & \none & \none & 1 & 1 & 1 \\
\none & \none & \none & 2 & 2 \\
1 & 1 & 3 \\
2 & 3
\end{ytableau}
\) \qquad \qquad 
\( d(M)= \)\hspace{1.2em}\raisebox{-.5\height}{\begin{tikzpicture}[x=.45cm,y=.45cm,baseline=(current bounding box.center),>=Stealth]
\filldraw[fill=gray!20,draw=black,line width=.6pt] (0,0) rectangle (1,-1);
\filldraw[fill=gray!20,draw=black,line width=.6pt] (1,0) rectangle (2,-1);
\filldraw[fill=gray!20,draw=black,line width=.6pt] (2,0) rectangle (3,-1);
\filldraw[fill=gray!20,draw=black,line width=.6pt] (3,0) rectangle (4,-1);
\filldraw[fill=gray!20,draw=black,line width=.6pt] (0,-1) rectangle (1,-2);
\filldraw[fill=gray!20,draw=black,line width=.6pt] (1,-1) rectangle (2,-2);
\filldraw[fill=gray!20,draw=black,line width=.6pt] (2,-1) rectangle (3,-2);
\filldraw[fill=gray!20,draw=black,line width=.6pt] (0,-2) rectangle (1,-3);
\filldraw[fill=gray!20,draw=black,line width=.6pt] (1,-2) rectangle (2,-3);
\filldraw[fill=gray!20,draw=black,line width=.6pt] (0,-3) rectangle (1,-4);
\draw[fill=white,line width=1pt] (0,-4) rectangle (1,-6);
\node at (0.500,-5.000) {1};
\draw[fill=white,line width=1pt] (1,-3) rectangle (2,-5);
\node at (1.500,-4.000) {1};
\draw[fill=white,line width=1pt] (4,0) rectangle (6,-1);
\node at (5.000,-0.500) {1};
\draw[fill=white,line width=1pt] (6,0) rectangle (8,-1);
\node at (7.000,-0.500) {1};
\draw[fill=white,line width=1pt] (8,0) rectangle (10,-1);
\node at (9.000,-0.500) {1};
\draw[fill=white,line width=1pt] (0,-6) rectangle (1,-8);
\node at (0.500,-7.000) {2};
\draw[fill=white,line width=1pt] (3,-1) rectangle (5,-2);
\node at (4.000,-1.500) {2};
\draw[fill=white,line width=1pt] (5,-1) rectangle (7,-2);
\node at (6.000,-1.500) {2};
\draw[fill=white,line width=1pt] (1,-5) rectangle (2,-7);
\node at (1.500,-6.000) {3};
\draw[fill=white,line width=1pt] (2,-2) rectangle (3,-4);
\node at (2.500,-3.000) {3};

\end{tikzpicture}}
\caption{An example of the construction of \( d(M) \) for
  \( M\in\LR(\lambda *\lambda,\mu) \)
  for \(\lambda=(3,2)\) and \( \mu=(5,3,2) \). The shaded region represents the staircase
  partition \( \delta_5=(4,3,2,1)\).}\label{fig:sigma-example}
\end{figure}
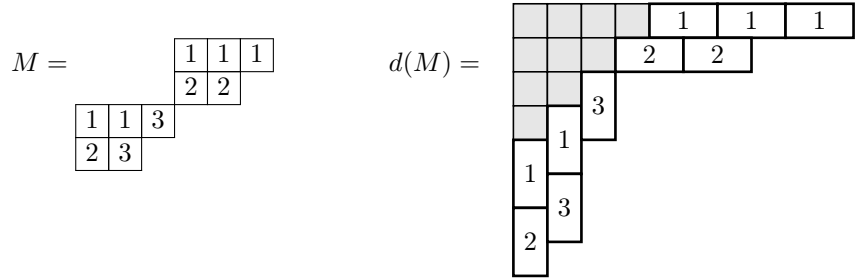

\begin{defn}\label{def:6}
  Suppose \( k\ge1 \) and \( \alpha \) is a partition with
  \( \delta_{k+1}\subseteq \alpha \). Let \( D \) be a semistandard domino
  tableau of shape \( \alpha/\delta_k \). We define \( \theta(D) \) as
  follows.

  \textbf{Step 1.} We first draw arrows between cells using the
  following rules.

  For each cell \( C=(i,j)\in \ZZ_{\ge1}^2 \) with
  \( j-i \equiv k \pmod 2 \) such that \( (\alpha/\delta_k)\cup \{C\} \)
  is a skew shape, we do the following. Let
  \( N=(i-1,j) \), \( E=(i,j+1) \), \( S=(i+1,j) \), and
  \( W=(i,j-1) \).
  \begin{description}
  \item[Case 1] Both \( N \) and \( W \) are contained in dominoes disjoint from \( C \).
    Let \( a \) and \( b \) be the labels of the dominoes containing \( N \) and \( W \), respectively.
    If \( a\ge b \), then draw an arrow from \( N \) to \( C \).
    If \( a< b \), then draw an arrow from \( W \) to \( C \).
  \item[Case 2] One of \( N \) and \( W \) is
    contained in a domino disjoint from \( C \) and the other is
    not contained in \( \alpha \). In this case, draw an arrow from the one of
    \( N \) and \( W \) that is contained in \( \alpha \) to
    \( C \).
  \item[Case 3] Both \( S \) and \( E \) are contained in dominoes disjoint from \( C \).
    Let \( a \) and \( b \) be the labels of the dominoes containing \( S \) and \( E \), respectively.
    If \( a\le b \), then draw an arrow from \( S \) to \( C \).
    If \( a> b \), then draw an arrow from \( E \) to \( C \).
  \item[Case 4] One of \( S \) and \( E \) is contained in a domino
    disjoint from \( C \) and the other is not contained in \( \alpha \). In
    this case, draw an arrow from the one of \( S \) and \( E \)
    that is contained in \( \alpha \) to \( C \).
  \end{description}
  If none of the cases above holds, we do not draw any arrow pointing
  to \( C \).

  \textbf{Step 2.} A \emph{chain} is a maximal sequence
  \( (\zeta_1,\dots,\zeta_m) \) of distinct dominoes such that there is
  an arrow from a cell in \( \zeta_{i+1} \) to a cell in
  \( \zeta_{i} \) for all \( i\in [m-1] \). If there is an arrow from
  a cell in \( \zeta_1 \) to a cell outside \( \alpha/\delta_k \),
  then we say that the chain is \emph{open}, and otherwise the chain
  is \emph{closed}.

  For every open chain \( (\zeta_1,\dots,\zeta_m) \), we modify the
  dominoes in this chain as follows. For each \( i\in [m] \), if there
  is an arrow from a cell \( A\in \zeta_{i} \) to a cell \( B \), then
  replace the domino \( \zeta_i \) by \( \{A,B\} \), carrying the label of
  \(\zeta_i\) with it.

 By van Leeuwen's construction, the above operation produces a
 semistandard domino tableau whose inner shape is \( \delta_{k-1} \).
 We denote this tableau by \( \theta(D) \).
\end{defn}

In the definition of \( \theta(D) \), only the dominoes in open chains
are modified, while those in closed chains remain fixed. 
By van Leeuwen's criterion, this operation preserves the Yamanouchi property
\cite[Section~5.1]{vanLeeuwen2000}. Hence, if \(D\) is Yamanouchi, then so is
\(\theta(D)\).

We are now ready to define the map \( \Phi_0 \).

\begin{defn}\label{def:7}
  Let \( \mu\vdash 2n \), \( \lambda\vdash n \), and
\( k=\lambda_1+\ell(\lambda) \).
We define the map
\[
  \Phi_0:\LR(\lambda*\lambda,\mu)\to \YDT(\lambda^\square,\mu)
\]
by \( \Phi_0(M) = \theta^{k-1}(d(M)) \).
\end{defn}

By \cite[Theorem~2.2.6]{vanLeeuwen2000}, the map \( \Phi_0 \) is a
bijection such that the weight of \( M\in \LR(\lambda*\lambda,\mu) \)
is equal to the weight of
\( \Phi_0(M)\in \YDT(\lambda^\square,\mu) \).

\begin{exam}\label{exam:phi0-running} 
  Consider the tableau \( M\in \LR(\lambda*\lambda,\mu) \) and
  \( d(M) \) in \Cref{fig:sigma-example}, where \( \lambda=(3,2) \)
  and \( \mu=(5,3,2) \). Then \(k=\lambda_1+\ell(\lambda) = 5\). Thus
  \( \Phi_0(M) = \theta^4(d(M)) \), as illustrated in
  \Cref{fig:phi0-w_c}. The tableau \(\Phi_0(M)\) lies in
  \(\YDT(\lambda^\square,\mu)\), as required, and has spin \(3\).
\end{exam}

\begin{figure}[t]
\hspace*{0.2em}
\noindent\makebox[6em][r]{\(D\)}~\(=\)~\raisebox{-.5\height}{\begin{tikzpicture}[x=.45cm,y=.45cm,baseline=(current bounding box.center),>=Stealth]
\filldraw[fill=gray!20,draw=black,line width=.6pt] (0,0) rectangle (1,-1);
\filldraw[fill=gray!20,draw=black,line width=.6pt] (1,0) rectangle (2,-1);
\filldraw[fill=gray!20,draw=black,line width=.6pt] (2,0) rectangle (3,-1);
\filldraw[fill=gray!20,draw=black,line width=.6pt] (3,0) rectangle (4,-1);
\filldraw[fill=gray!20,draw=black,line width=.6pt] (0,-1) rectangle (1,-2);
\filldraw[fill=gray!20,draw=black,line width=.6pt] (1,-1) rectangle (2,-2);
\filldraw[fill=gray!20,draw=black,line width=.6pt] (2,-1) rectangle (3,-2);
\filldraw[fill=gray!20,draw=black,line width=.6pt] (0,-2) rectangle (1,-3);
\filldraw[fill=gray!20,draw=black,line width=.6pt] (1,-2) rectangle (2,-3);
\filldraw[fill=gray!20,draw=black,line width=.6pt] (0,-3) rectangle (1,-4);
\draw[fill=white,line width=1pt] (0,-4) rectangle (1,-6);
\node at (0.500,-5.000) {1};
\draw[fill=white,line width=1pt] (1,-3) rectangle (2,-5);
\node at (1.500,-4.000) {1};
\draw[fill=white,line width=1pt] (4,0) rectangle (6,-1);
\node at (5.000,-0.500) {1};
\draw[fill=white,line width=1pt] (6,0) rectangle (8,-1);
\node at (7.000,-0.500) {1};
\draw[fill=white,line width=1pt] (8,0) rectangle (10,-1);
\node at (9.000,-0.500) {1};
\draw[fill=white,line width=1pt] (0,-6) rectangle (1,-8);
\node at (0.500,-7.000) {2};
\draw[fill=white,line width=1pt] (3,-1) rectangle (5,-2);
\node at (4.000,-1.500) {2};
\draw[fill=white,line width=1pt] (5,-1) rectangle (7,-2);
\node at (6.000,-1.500) {2};
\draw[fill=white,line width=1pt] (1,-5) rectangle (2,-7);
\node at (1.500,-6.000) {3};
\draw[fill=white,line width=1pt] (2,-2) rectangle (3,-4);
\node at (2.500,-3.000) {3};
\draw[-stealth,line width=1pt] (0.500,-4.500) -- (0.500,-3.500);
\draw[-stealth,line width=1pt] (1.500,-3.500) -- (1.500,-2.500);
\draw[-stealth,line width=1pt] (4.500,-0.500) -- (3.500,-0.500);
\draw[-stealth,line width=1pt] (6.500,-0.500) -- (5.500,-0.500);
\draw[-stealth,line width=1pt] (8.500,-0.500) -- (7.500,-0.500);
\draw[-stealth,line width=1pt] (0.500,-6.500) -- (0.500,-5.500);
\draw[-stealth,line width=1pt] (3.500,-1.500) -- (2.500,-1.500);
\draw[-stealth,line width=1pt] (5.500,-1.500) -- (4.500,-1.500);
\draw[-stealth,line width=1pt] (1.500,-5.500) -- (1.500,-4.500);
\draw[-stealth,line width=1pt] (2.500,-2.500) -- (3.500,-2.500);
\end{tikzpicture}}
\vspace{2.0em}
\noindent\makebox[3.5em][r]{\(\theta(D)\)}~\(=\)~\raisebox{-.5\height}{\begin{tikzpicture}[x=.45cm,y=.45cm,baseline=(current bounding box.center),>=Stealth]
\filldraw[fill=gray!20,draw=black,line width=.6pt] (0,0) rectangle (1,-1);
\filldraw[fill=gray!20,draw=black,line width=.6pt] (1,0) rectangle (2,-1);
\filldraw[fill=gray!20,draw=black,line width=.6pt] (2,0) rectangle (3,-1);
\filldraw[fill=gray!20,draw=black,line width=.6pt] (0,-1) rectangle (1,-2);
\filldraw[fill=gray!20,draw=black,line width=.6pt] (1,-1) rectangle (2,-2);
\filldraw[fill=gray!20,draw=black,line width=.6pt] (0,-2) rectangle (1,-3);
\draw[fill=white,line width=1pt] (0,-3) rectangle (1,-5);
\node at (0.500,-4.000) {1};
\draw[fill=white,line width=1pt] (1,-2) rectangle (2,-4);
\node at (1.500,-3.000) {1};
\draw[fill=white,line width=1pt] (3,0) rectangle (5,-1);
\node at (4.000,-0.500) {1};
\draw[fill=white,line width=1pt] (5,0) rectangle (7,-1);
\node at (6.000,-0.500) {1};
\draw[fill=white,line width=1pt] (7,0) rectangle (9,-1);
\node at (8.000,-0.500) {1};
\draw[fill=white,line width=1pt] (0,-5) rectangle (1,-7);
\node at (0.500,-6.000) {2};
\draw[fill=white,line width=1pt] (2,-1) rectangle (4,-2);
\node at (3.000,-1.500) {2};
\draw[fill=white,line width=1pt] (4,-1) rectangle (6,-2);
\node at (5.000,-1.500) {2};
\draw[fill=white,line width=1pt] (1,-4) rectangle (2,-6);
\node at (1.500,-5.000) {3};
\draw[fill=white,line width=1pt] (2,-2) rectangle (4,-3);
\node at (3.000,-2.500) {3};
\draw[-stealth,line width=1pt] (0.500,-3.500) -- (0.500,-2.500);
\draw[-stealth,line width=1pt] (1.500,-2.500) -- (1.500,-1.500);
\draw[-stealth,line width=1pt] (3.500,-0.500) -- (2.500,-0.500);
\draw[-stealth,line width=1pt] (5.500,-0.500) -- (4.500,-0.500);
\draw[-stealth,line width=1pt] (7.500,-0.500) -- (6.500,-0.500);
\draw[-stealth,line width=1pt] (0.500,-5.500) -- (0.500,-4.500);
\draw[-stealth,line width=1pt] (2.500,-1.500) -- (2.500,-2.500);
\draw[-stealth,line width=1pt] (4.500,-1.500) -- (3.500,-1.500);
\draw[-stealth,line width=1pt] (1.500,-4.500) -- (1.500,-3.500);
\draw[-stealth,line width=1pt] (3.500,-2.500) -- (4.500,-2.500);
\end{tikzpicture}}\par

\hspace*{-2.5em}
\noindent\makebox[6em][r]{\(\theta^2(D)\)}~\(=\)~\raisebox{-.5\height}{\begin{tikzpicture}[x=.45cm,y=.45cm,baseline=(current bounding box.center),>=Stealth]
\filldraw[fill=gray!20,draw=black,line width=.6pt] (0,0) rectangle (1,-1);
\filldraw[fill=gray!20,draw=black,line width=.6pt] (1,0) rectangle (2,-1);
\filldraw[fill=gray!20,draw=black,line width=.6pt] (0,-1) rectangle (1,-2);
\draw[fill=white,line width=1pt] (0,-2) rectangle (1,-4);
\node at (0.500,-3.000) {1};
\draw[fill=white,line width=1pt] (1,-1) rectangle (2,-3);
\node at (1.500,-2.000) {1};
\draw[fill=white,line width=1pt] (2,0) rectangle (4,-1);
\node at (3.000,-0.500) {1};
\draw[fill=white,line width=1pt] (4,0) rectangle (6,-1);
\node at (5.000,-0.500) {1};
\draw[fill=white,line width=1pt] (6,0) rectangle (8,-1);
\node at (7.000,-0.500) {1};
\draw[fill=white,line width=1pt] (0,-4) rectangle (1,-6);
\node at (0.500,-5.000) {2};
\draw[fill=white,line width=1pt] (2,-1) rectangle (3,-3);
\node at (2.500,-2.000) {2};
\draw[fill=white,line width=1pt] (3,-1) rectangle (5,-2);
\node at (4.000,-1.500) {2};
\draw[fill=white,line width=1pt] (1,-3) rectangle (2,-5);
\node at (1.500,-4.000) {3};
\draw[fill=white,line width=1pt] (3,-2) rectangle (5,-3);
\node at (4.000,-2.500) {3};
\draw[-stealth,line width=1pt] (0.500,-2.500) -- (0.500,-1.500);
\draw[-stealth,line width=1pt] (1.500,-1.500) -- (1.500,-0.500);
\draw[-stealth,line width=1pt] (2.500,-0.500) -- (2.500,-1.500);
\draw[-stealth,line width=1pt] (4.500,-0.500) -- (3.500,-0.500);
\draw[-stealth,line width=1pt] (6.500,-0.500) -- (5.500,-0.500);
\draw[-stealth,line width=1pt] (0.500,-4.500) -- (0.500,-3.500);
\draw[-stealth,line width=1pt] (2.500,-2.500) -- (1.500,-2.500);
\draw[-stealth,dashed,line width=1pt] (3.500,-1.500) -- (3.500,-2.500);
\draw[-stealth,line width=1pt] (1.500,-3.500) -- (2.500,-3.500);
\draw[-stealth,dashed,line width=1pt] (4.500,-2.500) -- (4.500,-1.500);
\end{tikzpicture}}
\vspace{2.0em}
\noindent\makebox[6em][r]{\(\theta^3(D)\)}~\(=\)~\raisebox{-.5\height}{\begin{tikzpicture}[x=.45cm,y=.45cm,baseline=(current bounding box.center),>=Stealth]
\filldraw[fill=gray!20,draw=black,line width=.6pt] (0,0) rectangle (1,-1);
\draw[fill=white,line width=1pt] (0,-1) rectangle (1,-3);
\node at (0.500,-2.000) {1};
\draw[fill=white,line width=1pt] (1,0) rectangle (2,-2);
\node at (1.500,-1.000) {1};
\draw[fill=white,line width=1pt] (2,0) rectangle (3,-2);
\node at (2.500,-1.000) {1};
\draw[fill=white,line width=1pt] (3,0) rectangle (5,-1);
\node at (4.000,-0.500) {1};
\draw[fill=white,line width=1pt] (5,0) rectangle (7,-1);
\node at (6.000,-0.500) {1};
\draw[fill=white,line width=1pt] (0,-3) rectangle (1,-5);
\node at (0.500,-4.000) {2};
\draw[fill=white,line width=1pt] (1,-2) rectangle (3,-3);
\node at (2.000,-2.500) {2};
\draw[fill=white,line width=1pt] (3,-1) rectangle (5,-2);
\node at (4.000,-1.500) {2};
\draw[fill=white,line width=1pt] (1,-3) rectangle (3,-4);
\node at (2.000,-3.500) {3};
\draw[fill=white,line width=1pt] (3,-2) rectangle (5,-3);
\node at (4.000,-2.500) {3};
\draw[-stealth,line width=1pt] (0.500,-1.500) -- (0.500,-0.500);
\draw[-stealth,dashed,line width=1pt] (1.500,-0.500) -- (2.500,-0.500);
\draw[-stealth,dashed,line width=1pt] (2.500,-1.500) -- (1.500,-1.500);
\draw[-stealth,line width=1pt] (3.500,-0.500) -- (3.500,-1.500);
\draw[-stealth,line width=1pt] (5.500,-0.500) -- (4.500,-0.500);
\draw[-stealth,line width=1pt] (0.500,-3.500) -- (0.500,-2.500);
\draw[-stealth,dashed,line width=1pt] (1.500,-2.500) -- (1.500,-3.500);
\draw[-stealth,line width=1pt] (4.500,-1.500) -- (5.500,-1.500);
\draw[-stealth,dashed,line width=1pt] (2.500,-3.500) -- (2.500,-2.500);
\draw[-stealth,line width=1pt] (3.500,-2.500) -- (3.500,-3.500);
\end{tikzpicture}}\par

\hspace*{-22em}
\noindent\makebox[6em][r]{\(\theta^4(D)\)} ~\(=\)~\raisebox{-.5\height}{\begin{tikzpicture}[x=.45cm,y=.45cm,baseline=(current bounding box.center),>=Stealth]
\draw[fill=white,line width=1pt] (0,0) rectangle (1,-2);
\node at (0.500,-1.000) {1};
\draw[fill=white,line width=1pt] (1,0) rectangle (2,-2);
\node at (1.500,-1.000) {1};
\draw[fill=white,line width=1pt] (2,0) rectangle (3,-2);
\node at (2.500,-1.000) {1};
\draw[fill=white,line width=1pt] (3,0) rectangle (4,-2);
\node at (3.500,-1.000) {1};
\draw[fill=white,line width=1pt] (4,0) rectangle (6,-1);
\node at (5.000,-0.500) {1};
\draw[fill=white,line width=1pt] (0,-2) rectangle (1,-4);
\node at (0.500,-3.000) {2};
\draw[fill=white,line width=1pt] (1,-2) rectangle (3,-3);
\node at (2.000,-2.500) {2};
\draw[fill=white,line width=1pt] (4,-1) rectangle (6,-2);
\node at (5.000,-1.500) {2};
\draw[fill=white,line width=1pt] (1,-3) rectangle (3,-4);
\node at (2.000,-3.500) {3};
\draw[fill=white,line width=1pt] (3,-2) rectangle (4,-4);
\node at (3.500,-3.000) {3};
\end{tikzpicture}}\par
\caption{Successive applications of the map \( \theta \) to \(D\).
Solid arrows indicate open chains and dashed arrows indicate closed chains.}
\label{fig:phi0-w_c}
\end{figure}
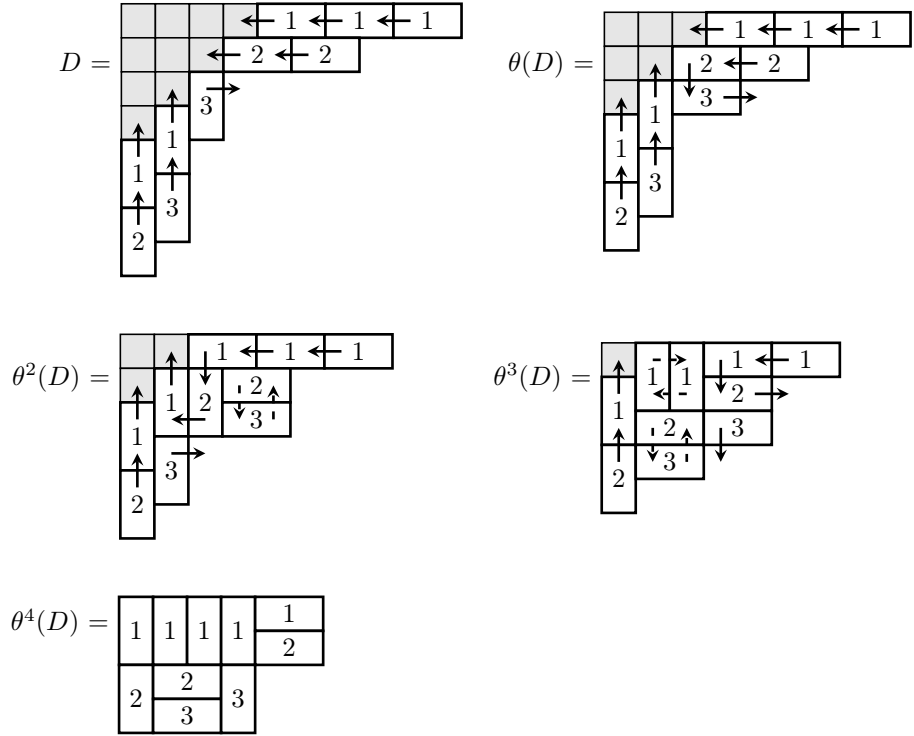

We complete the construction of the bijection \(\xi\) by assembling all the ingredients introduced above. 

\begin{defn}\label{def:xi}
  We define the map
  \[
\xi: \SYT^=(2n) \to \bigsqcup_{\substack{\lambda\vdash n \\ \mu\vdash 2n}}
    \SYT(\lambda) \times \YDT(\lambda^\square,\mu)
\]
as follows. For \(T\in\SYT^{=}(2n)\), let \( U=T_{[n]} \) and
\( S=\st(T_{[n+1,2n]}) \). Then
\[
\xi(T)=
\left(
U,\,
(\Phi_0\circ \Omega\circ \Psi_U)(S)
\right).
\]
Equivalently,
\[
\xi(T)= (\Phi'_0\circ \Omega'\circ \Psi' \circ \vartheta)(T),
\]
where \( \vartheta \) is the map in \eqref{eq:vartheta} and
\( \Psi'(U,A) = (U,\Psi_U(A)) \), \( \Omega'(U,B) = (U,\Omega(B)) \),
and \( \Phi'_0(U,C) = (U,\Phi_0(C)) \).
\end{defn}

\begin{exam} 
Let
\[
T=
\begin{ytableau}
1 & 2 & 4 & 7 & 9 \\
3 & 5 & 10 \\
6 & 8
\end{ytableau}
\in \SYT^{=}(10).
\]
By \Cref{exam:psi-running,exam:omega-running,exam:phi0-running}, we have \(\spin(T)=3\).
Since \( T\in\SYT_{(5,5)}((5,3,2)) \) and 
\(\spin(T)\equiv 5 \pmod{2} \), we have \( T\in \SYT_{(5,5)}^{\spin}((5,3,2)) \).
\end{exam}

Finally, we prove \Cref{lem:xi}.

\begin{proof}[Proof of \Cref{lem:xi}]
  Since the maps in \eqref{eq:vartheta} and \eqref{eq:three-maps} are
  bijections, so is \( \xi \). Consider \(T\in\SYT^{=}(2n)\). Letting
  \( U=T_{[n]} \) and \( S=\st(T_{[n+1,2n]}) \), we have
  \( \xi(T) = (U,D) \), where
  \( D= (\Phi_0\circ \Omega\circ \Psi_U)(S)\). Suppose
  \( \sh(T)=\mu \) and \( \sh(U) = \lambda \). Then
  \( S\in \SYT^U(\mu/\lambda) \), so we have
  \( D\in \YDT(\lambda^\square,\mu) \) by \eqref{eq:three-maps}.
  Therefore, the weight of \( D \) is equal to \( \mu=\sh(T) \), as desired.
\end{proof}

\section{Further directions}\label{sec:further-directions}

By \Cref{prop:prod-L-la}, to find a refined \( (\lambda,\mu) \)-Thrall
subset, it suffices to consider the case that \( \lambda \) is a
rectangle. In this paper, we have found a refined
\( (\lambda,\mu) \)-Thrall subset when \( \lambda \) is a rectangle
with two rows. Recall that the case when \( \lambda \) is a rectangle
with two columns is also known. Hence, a natural question is to
generalize this result to arbitrary rectangles.

\begin{problem}
  For \(n,k\ge 3\), find a \( ((k^n),\mu) \)-Thrall subset for each
  \(\mu\vdash kn\). In other words, for each \(\mu\vdash kn\), find a
  subset \( \mathcal{A}(\mu) \) of \( \SYT_{(k^n)}(\mu) \) such that
\[
\ch(L_{(k^n)}) = \sum_{\mu\vdash kn} |\mathcal{A}(\mu)|\, s_\mu.
\]
\end{problem}

The second problem is to understand the spin statistic more directly at
the level of standard Young tableaux. Currently, the
quantity \(\spin(T)\) is defined by passing through the chain of
bijections \( \Psi_U \), \( \Omega \), and \( \Phi_0 \) in
\eqref{eq:three-maps}.

\begin{problem}
  Find a direct description of the statistic \(\spin(T)\) in
  \Cref{def:4}. A weaker but still interesting goal is to describe the
  parity of \(\spin(T)\) directly.
\end{problem}

\section*{Acknowledgments}

This project was initiated at the KRIMS 2026 Winter School
``Symmetric Functions and Group Representations''. The authors thank
Soon-Yi Kang, Kyungbae Park, and Seunghyun Seo for their hospitality and support.

\end{document}